\documentclass{amsart}
\usepackage{amssymb}
\usepackage[all]{xy}

\usepackage{enumerate}
\usepackage{aliascnt}
\usepackage{hyperref}
\usepackage{graphicx}

\numberwithin{equation}{section}
\def \today{\number\day\space\ifcase\month\or   January\or February\or
   March\or April\or May\or June\or   July\or August \or September\or
   October\or November\or December\fi\   \number\year}

\newtheorem{lem}{Lemma}[section]

\newaliascnt{thmCt}{lem}
\newtheorem{thm}[thmCt]{Theorem}
\aliascntresetthe{thmCt}

\newaliascnt{corCt}{lem}
\newtheorem{cor}[corCt]{Corollary}
\aliascntresetthe{corCt}

\newaliascnt{prpCt}{lem}
\newtheorem{prp}[prpCt]{Proposition}
\aliascntresetthe{prpCt}

\theoremstyle{definition}

\newtheorem{ntn}[lem]{Notation}
\newaliascnt{dfnCt}{lem}
\newtheorem{dfn}[dfnCt]{Definition}
\aliascntresetthe{dfnCt}

\newaliascnt{rmkCt}{lem}
\newtheorem{rmk}[rmkCt]{Remark}
\aliascntresetthe{rmkCt}

\newaliascnt{qstCt}{lem}
\newtheorem{qst}[qstCt]{Question}
\aliascntresetthe{qstCt}

\newaliascnt{exaCt}{lem}
\newtheorem{exa}[exaCt]{Example}
\aliascntresetthe{exaCt}

\newcommand{\af}{\alpha}
\newcommand{\bt}{\beta}
\newcommand{\gm}{\gamma}
\newcommand{\dt}{\delta}
\newcommand{\ep}{\varepsilon}
\newcommand{\zt}{\zeta}
\newcommand{\et}{\eta}
\newcommand{\ch}{\chi}
\newcommand{\io}{\iota}
\newcommand{\te}{\theta}
\newcommand{\ld}{\lambda}
\newcommand{\sm}{\sigma}
\newcommand{\kp}{\kappa}
\newcommand{\ph}{\varphi}
\newcommand{\ps}{\psi}
\newcommand{\rh}{\rho}
\newcommand{\om}{\omega}
\newcommand{\ta}{\tau}

\newcommand{\Gm}{\Gamma}

\newcommand{\Q}{{\mathbb{Q}}}
\newcommand{\Z}{{\mathbb{Z}}}
\newcommand{\R}{{\mathbb{R}}}
\newcommand{\C}{{\mathbb{C}}}
\newcommand{\N}{{\mathbb{N}}}
\newcommand{\Nz}{\N \cup \{ 0 \}}

\newcommand{\bigast}{\scalebox{1.5}{$\ast$}}
\newcommand{\Bigast}{\scalebox{2.0}{$\ast$}}
\newcommand{\BFP}{\bigast}
\newcommand{\BigBFP}{\Bigast}

\newcommand{\ca}{$C^*$-algebra}
\newcommand{\ga}{$G$-algebra}
\newcommand{\stHm}{${}^*$-ho\-mo\-mor\-phism}
\newcommand{\ct}{continuous}
\newcommand{\lcg}{locally compact group}
\newcommand{\eqv}{equivariant}
\newcommand{\sj}{semipro\-jective}
\newcommand{\sjy}{semipro\-jectivity}
\newcommand{\pj}{projective}
\newcommand{\pjy}{projectivity}
\newcommand{\eqsj}{equivariantly semiprojective}
\newcommand{\eqpj}{equivariantly projective}
\newcommand{\sjpj}{(semi)pro\-jective}
\newcommand{\sjpjy}{(semi)pro\-jectivity}
\newcommand{\SINgp}{[SIN]-group}

\newcommand{\OA}[1]{{\mathcal{O}}_{#1}}
\newcommand{\CatG}{\mathcal{C}_G}
\newcommand{\CatH}{\mathcal{C}_H}
\newcommand{\Cb}{C_{\mathrm{b}}}
\newcommand{\Cu}{C_{\mathrm{u}}}

\newcommand{\Cru}{C_{\mathrm{ru}}}
\newcommand{\andeqn}{\, \, \, \, \, \, {\mbox{and}} \, \, \, \, \, \,}
\newcommand{\aGA}{\af \colon G \to \Aut (A)}

\newcommand{\GAa}{(G, A, \af)}
\newcommand{\GBb}{(G, B, \bt)}
\newcommand{\CGAa}{A \rtimes_{\af} G}

\DeclareMathOperator{\Ind}{Ind}
\DeclareMathOperator{\ord}{ord}
\DeclareMathOperator{\locdim}{locdim}
\DeclareMathOperator{\udim}{\Delta d}
\DeclareMathOperator{\LUCInd}{UInd}
\DeclareMathOperator{\id}{id}
\DeclareMathOperator{\ev}{ev}
\DeclareMathOperator{\Aut}{Aut}
\DeclareMathOperator{\im}{ran}
\DeclareMathOperator{\dist}{dist}
\DeclareMathOperator{\supp}{supp}
\DeclareMathOperator{\card}{card}

\title{Semiprojectivity with and without a group action}

\author{N.~Christopher Phillips}

\author{Adam P.~W.\  S{\o}rensen}

\author{Hannes Thiel}

\date{11~March 2014}

\address{Department of Mathematics, University of Oregon,
      Eugene OR 97403-1222, USA,
      and Department of Mathematics, University of Toronto,
      Room 6290, 40 St.\  George St., Toronto ON M5S 2E4, Canada.}

\email[]{ncp@darkwing.uoregon.edu}

\address{School of Mathematics and Applied Statistics,
    University of Wollongong,
    Wollongong NSW 2522, Australia.}

\email[]{sorensen@uow.edu.au}

\address{Mathematisches Institut,
    Fachbereich Mathematik und Informatik der Universit\"{a}t
      M\"{u}nster,
    Einsteinstrasse 62, 48149 M\"{u}nster, Germany.}

\email[]{hannes.thiel@uni-muenster.de}



\thanks{This material is based upon work of the first author
supported by the US National Science Foundation under Grants
DMS-0701076 and DMS-1101742
and by a visiting professorship at the Research Institute
for Mathematical Sciences at Kyoto University,
and on work of all three authors supported
by the Danish National Research Foundation
through the Centre for Symmetry and Deformation (DNRF92).
The second named author was supported by
The Danish Council
for Independent Research {\textbar} Natural Sciences.
}

\begin{document}

\begin{abstract}
The equivariant version of semiprojectivity
was recently introduced by the first author.
We study properties of this notion,
in particular its relation to ordinary semiprojectivity
of the crossed product
and of the algebra itself.

We show that equivariant semiprojectivity is preserved when the
action is restricted to a cocompact subgroup.
Thus, if a second countable compact group acts semiprojectively on a \ca{} $A$,
then $A$ must be semiprojective.
This fails for noncompact groups:
we construct a semiprojective action of $\Z$
on a nonsemiprojective \ca.

We also study equivariant projectivity and obtain analogous results,
however with fewer restrictions on the subgroup.
For example, if a discrete group acts projectively
on a \ca{} $A$,
then $A$ must be projective.
This is in contrast to the semiprojective case.

We show that the crossed product by a semiprojective action
of a finite group on a unital \ca{}
is a semiprojective \ca.
We give examples to show that this does not generalize
to all compact groups.
\end{abstract}

\maketitle


\indent
Equivariant \sjy{} was introduced in \cite{Phi2012},
by applying the usual definition of \sjy{}
to the category of unital \ga{s}
(\ca{s} with actions of the group~$G$)
with unital $G$-\eqv{} \stHm{s}.
See Definition~\ref{D_EqSj} below.
The purpose of \cite{Phi2012} was to show that certain
actions of compact groups
on various specific \ca{s} are \sj.
In particular, it is shown that any action of a second countable compact group
on a finite dimensional \ca{} is \sj,
and that for $n < \infty$, quasifree actions
of second countable compact
groups on the Cuntz algebras $\OA{n}$ are \sj.

In this paper we study \eqv{} \sjy{} more abstractly.
We also introduce \eqv{} \pjy{}
and carry out a parallel study of it.
We extend the definition to allow actions by general \lcg{s},
and we consider the nonunital version
of \eqv{} \sjy.

{}From the work in \cite{Phi2012}, it is not even clear
whether a \sj{} action of a noncompact group can exist.
One reason for skepticism was that the
trivial action of $\Z$ on $\C$ is not \sj,
as was shown by Blackadar~(\cite{BlaPrivate}).
We give a wide reaching generalization of this result
in \autoref{P_SjTriv},
by showing that if the trivial action of a group on a (nonzero)
\ca{} is \sj{} then the group must be compact.

There are, however, many nontrivial \sj{} (and even \pj{}) actions
of noncompact groups.
Indeed, given a countable discrete group $G$
and a \sj{} \ca{} $A$,
we show in \autoref{P_BernShift_EqSj_General}
that the free Bernoulli shift action of $G$
on the free product $\BFP_{g \in G} A$
is \eqsj.

Our main motivation was to understand how \eqv{} \sjy{}
(with group action) is related to \sjy{}
(without group action).
The following question naturally occurs:

\begin{qst}
\label{Q_Motivation1}
Assume that $\GAa$ is an \eqsj{} \ga{}
(Definition~\ref{D_EqSj} below).
Is $A$ \sj{} in the usual sense?
\end{qst}

We give a positive answer in \autoref{P_EqSj_implies_Sj}
under the assumption that $G$ is compact.
If we drop this assumption,
then the answer to the question may be negative.
Indeed, in \autoref{R:SjSubGp_notCpct} we construct
a \sj{} action of $\Z$ on a non\sj{} \ca.

\autoref{Q_Motivation1}
is a special case of a more natural question:

\begin{qst}
\label{Q_Motivation2}
Assume that $\GAa$ is a \ga{} that is \eqsj{} (\eqpj),
and let $H \leqslant G$ be a closed subgroup.
Is the restricted $H$-algebra $(H, A, \af |_H)$ \eqsj{} (\eqpj)?
\end{qst}

The two main results of this paper answer this question positively
under certain natural assumptions on the factor space $G / H$.
In the semiprojective case,
we get a positive answer (\autoref{P_SjSubGp})
if $H$ is cocompact, that is,
if $G / H$ is compact.
In the projective case, we get a positive answer
(\autoref{P_2X14PjSubGp}) if $H$ is compact or cocompact,
or if $G$ is a \SINgp{}
(meaning that the left and right uniformities on $G$ agree)
and $H$ is arbitrary.
These conditions are much less restrictive than in the \sj{} case.

This paper is organized as follows.
In \autoref{Sec_Intro} we give the definition of \eqv{} \sjy{}
(\autoref{D_EqSj}).
We also introduce \eqv{} \pjy{} (\autoref{D_EqPj}),
and we investigate the relation
between the unital and nonunital versions of these
definitions
(\autoref{P_SjNoUnit},
\autoref{P_SjUnitDiffCats},
and \autoref{P_SjUnitNoUnit}).

In \autoref{Sec_BernShift} we introduce free Bernoulli shifts,
and we show in \autoref{P_BernShift_EqSj_General}
that the free Bernoulli shifts
constructed from semiprojective \ca{s}
provide examples of \sj{} actions
of countable discrete groups.
We also study the orthogonal Bernoulli shift of $G$
on $\bigoplus_G A = C_0 (G, A)$ for \sjpj{} \ca{s}~$A$.
It turns out to be a much harder problem to determine
when this action is \sj,
and we give a positive answer,
for a semiprojective \ca~$A$, only
for finite cyclic groups of order $2^n$.
See \autoref{P_OrthShift}.

In \autoref{Sec_SjSubGp} we study
the semiprojective case of \autoref{Q_Motivation2}.
We give a positive answer (\autoref{P_SjSubGp})
when $G / H$ is compact.
It follows (\autoref{P_EqSj_implies_Sj})
that a second countable compact group can only act \sj{ly}
on a \ca{} that is \sj{} in the usual sense.
We show that this is not true in general,
by constructing in \autoref{R:SjSubGp_notCpct}
a \sj{} action of $\Z$ on a non\sj{} \ca.
The main ingredient in this section
is the induction functor (\autoref{D_2X08Ind}),
which assigns to each $H$-algebra an induced $G$-algebra.
We show that this functor is exact (\autoref{P_Ind_Exact})
and continuous (\autoref{P_Ind_Cts}).

In \autoref{Sec_PjSubGp} we study
\autoref{Q_Motivation2} in the projective case.
The main result is \autoref{P_2X14PjSubGp},
which gives a positive answer in considerable generality.
In particular, restriction to compact subgroups
preserves equivariant \pjy.
This shows that every \eqpj{} \ca{}
is (non\eqv{ly}) \pj{} (\autoref{P_EqPj_implies_Pj}),
which is in contrast to the \sj{} case
(\autoref{R:SjSubGp_notCpct}).

The main technique of this section
is an induction functor
which uses uniformly continuous functions; see \autoref{R_NCptInd}.
In \autoref{P_LUCInd-Exact}, we show that this functor is exact
if $H$ is compact or if $G$ is a \SINgp.
To prove this, we need conditions
under which uniformly continuous functions to a quotient \ca{}
can be lifted to uniformly continuous functions,
and in \autoref{P_UniformLifts}
we provide a satisfying answer
that might also be of independent interest.

In \autoref{Sec_SjCrPrd},
we study \sjy{} of crossed products.
In \autoref{P_CrPrdSj}, we show that for a discrete group~$G$
whose group \ca{} is \sj,
unital \sjy{} of an action $\aGA$ on a unital \ca{}
implies \sjy{} of the crossed product $\CGAa$.
\autoref{E_2X10_CptCP} shows that this can fail
when the group is compact but not finite.
At the end of \autoref{Sec_SjCrPrd},
we give counterexamples to several other
plausible relations
between \eqv{} \sjy{} for finite groups
and \sjy,
and state further open problems.

In \autoref{Sec_SjFixPt},
we study \sjy{} of fixed point algebras.
We show that for a saturated semiprojective action
of a finite group $G$ on a unital \ca~$A$,
the fixed point algebra $A^G$ is \sj{} (\autoref{P_2X10_SatSj}).
We show in \autoref{E_2X10_CptFix}
that this does not generalize to compact groups.
For a \sj{} action of a noncompact group,
we show in \autoref{P_FixPt} that the fixed point algebra is trivial.
Thus, the trivial action
of a noncompact group on a nonzero \ca{} is never \sj.
We therefore obtain a precise characterization
of when the trivial action of a group is \sjpj{} (\autoref{P_SjTriv}).


We use the following terminology and notation in this paper.
By a topological group we understand a group $G$ together with a
Hausdorff topology such that the map $(s, t) \mapsto s \cdot t^{- 1}$
is jointly \ct.
We mainly consider locally compact topological groups.
For such a group,
we denote its Haar measure by $\mu$.
By the Birkhoff-Kakutani theorem
(see Theorem 1.22 of \cite{MONZIP1955}),
$G$ is metrizable if and only if it is first countable.
Moreover,
in that case, the metric $d$ may be chosen to be left invariant,
that is,
$d (r s, r t) = d (s, t)$ for all $r, s, t \in G$.
We will always take our metrics to be left invariant.
We usually require $G$ to be second countable.

For a topological group~$G$,
by a {\emph{$G$-algebra}} we understand a triple $\GAa$
in which $A$ is a \ca{}
and $\af \colon G \to \Aut (A)$ is a \ct{} action of~$G$ on~$A$.
Continuity means that for each $a \in A$ the map
$s \mapsto \af_s (a)$ is \ct.
(Such an action is also called strongly continuous.)

By a {\emph{$G$-morphism}}
between two $G$-algebras $\GAa$ and $\GBb$
we mean a $G$-equivariant \stHm, that is,
a \stHm{} $\ph \colon A \to B$ such that
$\bt_s \circ \ph = \ph \circ \af_s $ for each $s \in G$.
We say that a \ga{} $\GAa$ is {\emph{separable}}
if $A$ is a separable \ca{} and $G$ is second countable
(hence also metrizable).

Given a $G$-algebra $\GAa$,
we denote by $A^G$ its fixed point algebra
\[
A^G
 = \big\{ a \in A \colon
   {\mbox{$\af_s (a) = a$ for all $s \in G$}} \big\}
\]
(even when $G$ is not compact),
and by $\CGAa$ the (maximal) crossed product of $\GAa$.

If $A$ is a \ca,
we denote by $A^{+}$ its unitization
(adding a new identity even if $A$ already has an identity).
We let ${\widetilde{A}}$ be $A^{+}$ when $A$ is not unital
and be $A$ when $A$ is unital.
If $A$ is a $G$-algebra,
then $A^{+}$ and ${\widetilde{A}}$ are both $G$-algebras
in an obvious way.

Subalgebras of \ca{s} are always assumed to be $C^*$-subalgebras,
and ideals are always closed and two sided.

We use the convention $\N = \{ 1, 2, \ldots \}$.


\section{Equivariant semiprojectivity and equivariant projectivity}
\label{Sec_Intro}

\indent
In this section we recall the definition of \eqv{} \sjy.
We also give a nonunital version,
and we will see in Lemmas
\ref{P_SjNoUnit} and
\ref{P_SjUnitDiffCats}
and \autoref{P_SjUnitNoUnit} how the two variants are related.
We also introduce \eqv{} \pjy.

The unital case of the following definition is
Definition~1.1 of \cite{Phi2012}.


\begin{dfn}
\label{D_EqSj}
A separable $G$-algebra $\GAa$ is called
{\emph{\eqsj{}}}
if whenever $(G, C, \gm)$ is a \ga,
$J_1 \subset J_2 \subset \cdots$ is an increasing sequence
of $G$-invariant ideals in~$C$,
$J = {\overline{\bigcup_{n = 1}^{\infty} J_n}}$,
$\pi_n \colon C/J_n \to C/J$ is the quotient \stHm{}
for $n \in \N$,
and $\ph \colon A \to C/J$ is a $G$-morphism,
then there exist $n \in \N$
and a $G$-morphism $\ps \colon A \to C/J_n$
such that $\pi_n \circ \ps = \ph$.

When no confusion can arise, we say that $A$ is \eqsj,
or that $\af$ is \sj.

We say that a separable unital \ga{} $\GAa$ is
{\emph{\eqsj{} in the unital category}}
if the same condition holds,
but under the additional assumption that $C$ and $\ph$ are unital,
and the additional requirement that one can choose $\ps$ to be unital.
\end{dfn}

The lifting problem of the definition
means that in the right diagram that appears below \autoref{D_EqPj},
the solid arrows are given,
and $n$ and $\ps$ are supposed to exist which make the
diagram commute.


\begin{dfn}
\label{D_EqPj}
A $\GAa$ is called {\emph{\eqpj{}}}
if whenever $(G, C, \gm)$ is a \ga,
$J$ is a $G$-invariant ideal in~$C$
with quotient \stHm{} $\pi \colon C \to C/J$,
and $\ph \colon A \to C/J$ is a $G$-morphism,
then there exists a $G$-morphism $\ps \colon A \to C$
such that $\pi \circ \ps = \ph$.

When no confusion can arise,
we say that $A$ is \eqpj, or that $\af$ is \pj.

We say that a unital \ga{} is
\emph{\eqpj{} in the unital category}
if the same condition holds,
but under the additional assumption that $C$ and $\ph$ are unital,
and the additional requirement that one can choose $\ps$ to be unital.
\end{dfn}


\noindent \hspace{-0.25cm}
\begin{tabular}{p{6cm}p{0.1cm}lp{0.1cm}l}
\hspace{0.3cm}
The lifting problem of the definition
means that the left diagram on the right can be completed.
Again, the solid arrows are given,
and $\ps$ is supposed to exist which makes the
diagram commute.

\hspace{0.3cm}
When working with \sjy{} and \pjy,
it is often convenient,
in the notation of \autoref{D_EqSj} and \autoref{D_EqPj},
to require that the map $\ph$ be an isomorphism.
& &
\makebox{
\xymatrix{
& C \ar@{->}[d]^{\pi} \\
A \ar[r]_-{\ph} \ar@{-->}[ru]^{\ps} & C / J.
} }
& &
\makebox{ \xymatrix{
& C \ar[d] \\
& C / J_n \ar[d]^{\pi_n} \\
A \ar[r]_-{\ph} \ar@{-->}[ru]^{\ps} & C / J.
} }
\\
\multicolumn{5}{p{\textwidth}}{
\hspace{0.3cm}
This can also be done in the equivariant case.
The proof follows that of Proposition~2.2 of \cite{Bla2004}.
We give the proof since \cite{Bla2004} is a survey article
and its proof omits some details.
}
\end{tabular}


\begin{lem} \label{L_2X05IdAlgInt}
Let $A$ be a \ca,
let $B \subset A$ be a $C^*$-subalgebra,
and let $I_1 \subset I_2 \subset \cdots \subset A$ be ideals.
Then
$B \cap {\overline{\bigcup_{n = 1}^{\infty} I_n}}
  = {\overline{\bigcup_{n = 1}^{\infty} (B \cap I_n)}}$.
\end{lem}


\begin{proof}
We have
$\bigcup_{n = 1}^{\infty} (B \cap I_n)
  \subset B \cap {\overline{\bigcup_{n = 1}^{\infty} I_n}}$,
so
${\overline{\bigcup_{n = 1}^{\infty} (B \cap I_n)}}
  \subset B \cap {\overline{\bigcup_{n = 1}^{\infty} I_n}}$.

For the reverse,
let $b \in B$ and suppose that
$b \not \in {\overline{\bigcup_{n = 1}^{\infty} (B \cap I_n)}}$.
Let $\rh$ be the norm of the image of $b$ in
$B \big/
  {\overline{\bigcup_{n = 1}^{\infty} (B \cap I_n)}}$.
For $n \in \N$,
let $\kp_n \colon B \to B / (B \cap I_n)$
and $\pi_n \colon A \to A / I_n$
be the quotient maps.
Then $\| \kp_n (b) \| \geq \rh$ for all $n \in \N$.
The inclusion $\io \colon B \to A$ induces injective
\stHm{s} $\io_n \colon B / (B \cap I_n) \to A / I_n$
such that $\pi_n \circ \io = \io_n \circ \kp_n$.
Since $\io_n$ is isometric,
we have
\[
\| (\pi_n \circ \io) (b) \|
 = \| \io_n (\kp_n (b)) \|
 = \| \kp_n (b) \|
 \geq \rh,
\]
whence $\dist (b, I_n) \geq \rh$.
This is true for all $n \in \N$,
so $b \not \in {\overline{\bigcup_{n = 1}^{\infty} I_n}}$.
\end{proof}


\begin{prp}
\label{P_SjInjSurj}
Let $(G, A, \alpha)$ be a \ga{}
(separable for the statements involving \sjy).
As for usual \sjpjy,
the definitions of equivariant \sjy{} (\autoref{D_EqSj})
and of equivariant \pjy{} (\autoref{D_EqPj}) for $(G, A, \alpha)$,
in both the unital and nonunital categories,
are unchanged if,
in the notation of these definitions, we require one or both
of the following:
\begin{enumerate}
\item\label{P_SjInjSurj-In}
$\ph$ is injective.
\item\label{P_SjInjSurj-Su}
$\ph$ is surjective.
\end{enumerate}
\end{prp}


\begin{proof}
We give the proof for equivariant \sjy{} in the unital category.
The other cases are similar but slightly simpler.

Throughout,
let the notation be as in \autoref{D_EqSj}.

We first prove
the result for the restriction~(\ref{P_SjInjSurj-In}).
So assume that $C$, $J_1 \subset J_2 \subset \cdots \subset C$,
$J$,
quotient maps $\pi_n \colon C / J_n \to C / J$,
and $\ph \colon A \to C / J$, all as in \autoref{D_EqSj},
are given.
The following diagram shows the algebras and maps to be constructed:
%
\[
\xymatrix{
    & A \oplus C \ar[r]^-{\rh} \ar[d] & C \ar[d]  \\
    & A \oplus C / J_n \ar[r]^-{\rh_n} \ar[d]^{\id_A \oplus \pi_n}
        & C / J_n \ar[d]^{\pi_n} \\
A \ar[r]^<<<<<<{\mu} \ar@{-->}[ru]^{\nu} \ar@/_1.5pc/[rr]_{\ph}
    & A \oplus C / J \ar[r]^-{\rh_{\infty}}  & C / J.
}
\]
%
Equip $A \oplus C$,
$A \oplus C / J_n$ for $n \in \N$,
and $A \oplus C / J$
with the direct sum actions of~$G$.
Let $\rh \colon A \oplus C \to C$,
$\rh_n \colon A \oplus C / J_n \to C / J_n$ for $n \in \N$,
and $\rh_{\infty} \colon A \oplus C / J \to C / J$
be the projections on the second summand.
Define $\mu \colon A \to A \oplus C / J$ by
$\mu (a) = (a, \ph (a))$
for $a \in A$.
Then $\mu$ is a unital injective $G$-morphism
such that $\rh_{\infty} \circ \mu = \ph$.
By hypothesis,
there are $n \in \N$
and a unital $G$-morphism $\nu \colon A \to A \oplus C / J_n$
such that $(\id_A \oplus \pi_n) \circ \nu = \mu$.
Then the map $\ps = \rh_n \circ \nu$ is a unital $G$-morphism
such that $\pi_n \circ \nu = \ph$.
This completes the proof of~(\ref{P_SjInjSurj-In}).

We now prove that the condition is equivalent when both
restrictions (\ref{P_SjInjSurj-Su}) and~(\ref{P_SjInjSurj-In})
are applied.
It follows that the condition is also equivalent
when only~(\ref{P_SjInjSurj-Su})
is applied.

So let the notation be as before,
and assume in addition that $\ph$ is injective.
The following diagram shows the algebras and maps to be constructed:
%
\[
\xymatrix{
    & D \ar[r]^{\rh} \ar[d] & C \ar[d] \ar@/_-2pc/[dd]^{\pi} \\
    & D / I_n \ar[r]^{\rh_n} \ar[d]^{\kp_n} & C / J_n \ar[d]^{\pi_n} \\
A \ar[r]^{\mu} \ar@{-->}[ru]^{\nu} \ar@/_1.5pc/[rr]_{\ph}
    & D / I \ar[r]^{\rh_{\infty}}  & C / J.
}
\]
%

Let $\pi \colon C \to C / J$ be the quotient map.
Set $D = \pi^{- 1} (\ph (A))$ and $I = D \cap J$.
For $n \in \N$ set $I_n = D \cap J_n$
and let $\kp_n \colon D / I_n \to D / I$
be the quotient map.
Then ${\overline{\bigcup_{n = 1}^{\infty} I_n}} = I$
by \autoref{L_2X05IdAlgInt}.

Let $\rh \colon D \to C$ be the inclusion.
Then $\rh$ drops to a \stHm{} $\rh_n \colon D / I_n \to C / J_n$
for every $n \in \N$,
and to a \stHm{} $\rh_{\infty} \colon D / I \to C / J$.
All these maps are injective.
Clearly the range of $\ph$ is contained in $\rh_{\infty} (D / I)$,
so there is a \stHm{} $\mu \colon A \to D / I$
such that $\rh_{\infty} \circ \mu = \ph$.
This \stHm{} is injective because $\ph$ is
and surjective by the definition of~$D$.
The hypothesis implies that
there are $n \in \N$ and $\nu \colon A \to D / I_n$
such that $\kp_n \circ \nu = \mu$.
Then the map $\ps = \rh_n \circ \nu$ satisfies $\pi_n \circ \nu = \ph$.
\end{proof}


It is a standard result in the theory of \sjy{}
(contained in Lemma 14.1.6 and Theorem 14.1.7 of \cite{Lor1997})
that for a nonunital \ca~$A$ the following are equivalent:
\begin{enumerate}
\item\label{X907SyU-NU}
$A$ is \sj.
\item\label{X907SyU-InU}
$\widetilde{A}$ is \sj{} in the unital category.
\item\label{X907SyU-Uniz}
$\widetilde{A}$ is \sj.
\end{enumerate}
In the equivariant case,
the equivalence of all three conditions holds when the group~$G$ is
compact.
The proof of the analog of the implications
from (\ref{X907SyU-NU}) to~(\ref{X907SyU-Uniz})
and from (\ref{X907SyU-InU}) to~(\ref{X907SyU-Uniz})
breaks down when the trivial action of $G$ on~$\C$
is not \sj{} in the nonunital category,
but the remaining implications hold in general.
The trivial action on $\C$ is always
\sj{} in the {\emph{unital}} category,
but we will show in \autoref{P_SjTriv} that it is \sj{}
in the nonunital category only if $G$ is compact.


\begin{lem} \label{P_SjNoUnit}
Let $\GAa$ be a separable \ga,
with $A$ nonunital.
Then $A$ is \eqv{ly} \sjpj{} if and only if
${\widetilde{A}}$ is \eqv{ly} \sjpj{}
in the unital category.
\end{lem}


\begin{proof}
We give the proof for \eqv{} \sjy.
The proof for \eqv{} projectivity is similar but easier.
We use the notation of \autoref{D_EqSj}.

Since $A$ is nonunital,
we have ${\widetilde{A}} = A^{+}$.

First assume that $A$ is \eqsj,
and that $C$ and $\ph \colon A^{+} \to C / J$
are unital.
By \eqv{} \sjy{} of~$A$,
there are $n \in \N$ and $\ps_0 \colon A \to C / J_n$
such that $\pi_n \circ \ps_0 = \ph |_A$.
Then the formula
$\ps (a + \ld \cdot 1_{A^{+}}) = \ps_0 (a) + \ld \cdot 1_{C / J_n}$,
for $a \in A$ and $\ld \in \C$,
defines a $G$-morphism $\ps \colon A^{+} \to C / J$
such that $\pi_n \circ \ps = \ph$.
We have shown that ${\widetilde{A}}$ is \eqsj{}
in the unital category.

Now assume that $A^{+}$ is \eqsj{} in the unital category,
and in the notation of \autoref{D_EqSj}
take $C$ and $\ph \colon A \to C / J$ to be not necessarily unital.
We have obvious isomorphisms $C^{+} / J_n \cong (C / J_n)^{+}$
for $n \in \N$ and $C^{+} / J \cong (C / J)^{+}$.
(We add a new unit even if $C$ is already unital.)
Let $\nu_n \colon C^{+} / J_n \to \C$
for $n \in \N$,
and  $\nu_{\infty} \colon C^{+} / J \to \C$,
be the maps associated with the unitizations.
Define a unital $G$-morphism $\ph^{+} \colon A^{+} \to C^{+} / J$
by $\ph^{+} (a + \ld \cdot 1_{A^{+}})
  = \ph (a) + \ld \cdot 1_{C^{+} / J}$
for $a \in A$ and $\ld \in \C$.
For $n \in \N$,
similarly define $\pi_n^+{} \colon C^{+} / J_n \to C^{+} / J$,
giving $\nu_{\infty} \circ \pi_n^{+} = \nu_n$.
By hypothesis,
there are $n \in \N$ and $\ps_0 \colon A^{+} \to C^{+} / J_n$
such that $\pi_n^{+} \circ \ps_0 = \ph^{+}$.

We claim that $\ps_0 (A) \subset C / J_n$.
We have
\[
\nu_n \circ \ps_0
  = \nu_{\infty} \circ \pi_n^{+} \circ \ps_0
  = \nu_{\infty} \circ \ph^{+},
\]
which vanishes on~$A$.
The claim follows.
So $\ps = \ps_0 |_A \colon A \to C / J_n$
is a $G$-morphism such that $\pi_n \circ \ps = \ph$.
\end{proof}


\begin{lem}\label{P_SjUnitDiffCats}
Let $\GAa$ be a separable \ga,
with $A$ unital.
If $A$ is \eqv{ly} \sj,
then $A$ is \eqv{ly} \sj{} in the unital category.
If $G$ is compact, then the converse also holds.
\end{lem}


\begin{proof}
The proof is essentially the same
as that of Lemma 14.1.6 of \cite{Lor1997}.
In the first paragraph of the proof there,
$B_l$ should be $C_l$
and it is $1 - \ph_l (1)$, not $\ph_l (1) - 1$,
that is a projection.
In the second paragraph of the proof there,
we need the equivariant version of Lemma 14.1.5 of \cite{Lor1997};
it follows from Corollary~1.9 of \cite{Phi2012}.
\end{proof}


For compact groups,
we now obtain the analog of the equivalence of the first two parts
in Theorem~14.1.7 of \cite{Lor1997}.


\begin{prp}
\label{P_SjUnitNoUnit}
Let $G$ be a second countable compact group,
and let $A$ be a separable \ga.
Then $A$ is \eqsj{} if and only if ${\widetilde{A}}$ is.
\end{prp}


\begin{proof}
Combine \autoref{P_SjNoUnit} and \autoref{P_SjUnitDiffCats}.
\end{proof}


The paper \cite{Phi2012}
contains many examples of \eqsj{} \ca{s}.
In particular, it is shown that for a \sj{} \ca{} $A$
and a second countable compact group $G$,
the trivial action of $G$ on $A$ is \sj{}
(Corollary~1.9 of \cite{Phi2012}).
In the same way one may prove the analog for the \pj{} case,
and we include the short argument for completeness.
The following lemma is an immediate consequence
of Lemma~1.6 of \cite{Phi2012}.


\begin{lem}
\label{P_FixToFix}
Let $G$ be a compact group
and let $\pi \colon A \to B$
be a surjective $G$-morphism of \ga{s}.
Then the restriction of $\pi$ to the fixed point algebras is surjective,
that is, $\pi (A^G) = B^G$.
\end{lem}


\begin{lem}
\label{P_PjTrivAct}
Let $G$ be a second countable compact group, let $A$ be a \pj{} \ca,
and let $\io \colon G \to \Aut (A)$ be the trivial action.
Then $(G, A, \io)$ is \eqpj.
\end{lem}


\begin{proof}
By \autoref{P_SjInjSurj},
it is enough to show that
any surjective $G$-morphism $\pi \colon C \to A$
has an \eqv{} right inverse~$\ps$.
Since $G$ acts trivially on $A$, we have $A^G = A$,
and then $\pi (C^G) = A$ by \autoref{P_FixToFix}.
We can now use \pjy{} of $A$ to get a \stHm{}
$\gm \colon A \to C^G$
such that $\pi \circ \gm = \id$.
Let $\ps$ be the composition of $\gm$
with the inclusion of $C^G$ in~$C$.
\end{proof}


\begin{rmk}\label{R_2X10_Unl}
The statement of \autoref{P_FixToFix} can fail
if $G$ is not compact.
In fact, for every (second countable) noncompact group $G$,
we construct in the proof of \autoref{P_FixPt}
a surjective $G$-morphism $\pi \colon A \to B$ such that
$\pi (A^G) \neq B^G$.
\end{rmk}


So far, we have only seen \sj{} actions of compact groups.
In the next section we will show that every countable discrete group
even admits \pj{} actions.


\section{Free and orthogonal Bernoulli shifts}
\label{Sec_BernShift}

\indent
In this section, we introduce for every countable discrete group
$G$ and \ca~$A$ a natural action,
called the free Bernoulli shift, of $G$ on
the full free product $\BFP_{g \in G} A$.
We show (\autoref{P_BernShift_EqSj_General})
that this action is \sjpj{} if $A$ is
(non\eqv{ly}) \sj.

We also investigate the orthogonal Bernoulli shift of $G$
on $\bigoplus_{g \in G} A$, that is,
the translation action of $G$ on $C_0 (G, A)$.
It seems to be much more difficult to determine
when this action is \sjpj.
In \autoref{P_OrthShift}
we give a positive answer for the special case
that $G$ is finite cyclic of order~$2^n$.


We use the following notation,
roughly as before Remark~3.1.2 of \cite{Lor1997},
for the universal \ca{} on countably many contractions.


\begin{ntn}\label{N_2X07UnivAlg}
Set
\[
{\mathcal{F}}_{\infty}
 = C^* \big\langle z_1, z_2, \ldots \mid
     {\mbox{$\| z_j \| \leq 1$ for $j \in \N$}} \big\rangle,
\]
the universal \ca{} on generators $z_1, z_2, \ldots$
with relations $\| z_j \| \leq 1$ for $j \in \N$.

Let $G$ be a countable discrete group.
Set $P_G = \BFP_{g \in G} {\mathcal{F}}_{\infty}$.
For $s \in G$ let
$\io_s \colon {\mathcal{F}}_{\infty}
   \to \BFP_{g \in G} {\mathcal{F}}_{\infty}$
be the map which sends ${\mathcal{F}}_{\infty}$
to the copy of ${\mathcal{F}}_{\infty}$
in~$P_G$
indexed by~$s$.
We identify $P_G$ with
\[
C^* \big\langle \{ z_{s, k} \colon {\mbox{$s \in G$ and $k \in \N$}} \}
    \mid
      {\mbox{$\| z_{s, k} \| \leq 1$ for $s \in G$ and $k \in \N$}}
     \big\rangle,
\]
in such a way that
$\io_s (z_k) = z_{s, k}$ for $s \in G$ and $k \in \N$.
\end{ntn}


\begin{lem}\label{L_2X07ExistFBS}
Let $A$ be a \ca{} and let $G$ be a discrete group.
For $s \in G$ let $\io_{A, s} \colon A \to \BFP_{g \in G} A$
be the map which sends $A$
to the copy of $A$
in $\BFP_{g \in G} A$
indexed by~$s$.
Then there exists a unique action
$\ta^A \colon G \to \Aut \big( \BFP_{g \in G} A \big)$
such that $\ta^A_g (\io_{A, s} (a)) = \io_{A, g s} (a)$
for all $g, s \in G$ and $a \in A$.
For every \stHm{} $\ph \colon A \to B$ between \ca{s} $A$ and~$B$,
the corresponding \stHm{}
$\BFP_{g \in G} \; \ph \colon \BFP_{g \in G} A \to \BFP_{g \in G} B$
is equivariant.
Moreover,
for every $G$-algebra $(G, C, \gm)$
and every \stHm{} $\ph \colon A \to C$,
there is a unique $G$-morphism
$\ps \colon \BFP_{g \in G} A \to C$
such that
$\ps \circ \io_{A, s} = \gm_s \circ \ph$ for all $s \in G$.
\end{lem}


\begin{proof}
This is immediate.
\end{proof}


\begin{dfn}\label{D_2X07FBS}
Let $A$ be a separable \ca{} and let $G$ be a countable discrete group.
The action $\ta^A$ of \autoref{L_2X07ExistFBS}
is called the {\emph{free Bernoulli shift based on~$A$}}.
If $A = {\mathcal{F}}_{\infty}$,
so that $\BFP_{g \in G} A = P_G$,
we call it the {\emph{universal free Bernoulli shift}},
and denote it by~$\ta$.
\end{dfn}


Following Notation~\ref{N_2X07UnivAlg},
we have $\ta_s (z_{t, k}) = z_{s t, k}$ for $s, t \in G$ and $k \in \N$.

Any separable \ca~$A$ is a quotient of~${\mathcal{F}}_{\infty}$.
This is just the fact that
$A$ contains a countable set of contractive generators.
Similarly,
the action $\ta \colon G \to \Aut (P_G)$
is universal for all $G$-actions, that is,
for every separable $G$-algebra $A$
there exists a surjective $G$-morphism $P_G \to A$.


\begin{prp}
\label{P_BernShift_EqSj_General}
Let $A$ be a separable \ca{} and let $G$ be a countable discrete group.
If $A$ is \sjpj,
then the free Bernoulli shift of $G$ based on $A$ is \sjpj.
\end{prp}


\begin{proof}
We give the proof when $A$ is \pj.
The \sj{} case is very similar, but has bigger diagrams.

Let $\GBb$ and $(G, D, \dt)$ be $G$-algebras,
let $\pi \colon B \to D$ be a surjective $G$-morphism,
and let $\ph \colon \BFP_{g \in G} A \to D$ be a $G$-morphism.
We find
a $G$-morphism $\ps \colon \BFP_{g \in G} A \to B$
such that $\pi \circ \ps = \ph$.

Since $A$ is \pj, there is a \stHm{} $\ps_0 \colon A \to B$
such that $\pi \circ \ps_0 = \ph \circ \io_1$.
By universality of $\BFP_{g \in G} A$ (\autoref{L_2X07ExistFBS}),
there is a \stHm{}
$\ps \colon \BFP_{g \in G} A \to B$
such that $\ps \circ \io_{A, s} = \bt_s \circ \ps_0$
for all $s \in G$.
The following diagram shows some of the maps:
\[
        \xymatrix{
                & & B \ar[d]^{\pi} \\
                A \ar[r]_-{\io_1} \ar@/^1em/[urr]^{\ps_0} &
                \BigBFP_{g \in G} A \ar[r]_-{\ph} \ar@{-->}[ur]_{\ps} &
                D.
        }
\]
It remains to show that $\pi \circ \ps = \ph$
and that $\ps$ is $G$-equivariant.

Let $t \in G$.
Using equivariance of $\pi$ at the second step
and equivariance of $\ph$ at the fourth step, we get
\[
\pi \circ \ps \circ \io_{A, t}
          = \pi \circ \bt_t \circ \ps_0
          = \dt_t \circ \pi \circ \ps_0
          = \dt_t \circ \ph \circ \io_1
          = \ph \circ \ta_t \circ \io_1
          = \ph \circ \io_{A, t}.
\]
Since this is true for all $t \in G$,
and since $\bigcup_{t \in G} \io_{A, t} (A)$
generates $\BFP_{g \in G} A$,
it follows that $\pi \circ \ps = \ph$.

To see that $\ps$ is equivariant, let $s, t \in G$.
We compute:
\[
\bt_s \circ \ps \circ \io_{A, t}
        = \bt_s \circ \bt_t \circ \ps_0
        = \bt_{st} \circ \ps_0
        = \ps \circ \io_{A, st}
        = \ps \circ \ta_s \circ \io_{A, t}.
\]
For the same reason as in the previous paragraph,
it follows that $\bt_s \circ \ps = \ps \circ \ta_s$,
and so $\ps$ is $G$-equivariant.
\end{proof}


\begin{rmk}
Let $G$ be countable discrete group.
The universal $G$-algebra $P_G$
is (non\eqv{ly}) \pj,
since it is isomorphic to ${\mathcal{F}}_{\infty}$.
We can use this to show that if $\af \colon G \to \Aut (A)$
is a \pj{} action, then $A$ must be \pj.
This is a special case of
\autoref{P_EqPj_implies_Pj}.

Using the universal property of $P_G$ and separability
of $A$, we can find a surjective $G$-morphism
$\rho \colon P_G \to A$.
Since $\af$ is \eqpj, we can find a $G$-morphism
$\lambda \colon A \to P_G$ such that $\rh \circ \lambda = \id_A$.
If now $\ph \colon A \to C / J$ is a \stHm,
there is a \stHm{} $\ps \colon P_G \to C$ which lifts $\ph \circ \rh$.
Then $\ps \circ \lambda$ lifts~$\ph$.

A more involved argument,
which we do not give here,
gives a similar result for finite groups and \sjy.
\end{rmk}


We now turn to what we call the orthogonal Bernoulli shift.

\begin{dfn}\label{R:OrthShift}
Let $A$ be a separable \ca{} and let $G$ be a countable discrete group.
The {\emph{orthogonal Bernoulli shift based on~$A$}}
is the action
$\sm^A \colon G \to \Aut ( C_0 (G, A))$
given by $\sm^A_s (a) (t) = a (s^{- 1} t)$ for $a \in C_0 (G, A)$
and $s, t \in G$.
\end{dfn}

We think of $C_0 (G, A)$ as $\bigoplus_{g \in G} A$.
Then the automorphism $\sm^A_s$ sends the summand indexed by $t \in G$
to the summand indexed by $s t$.

By analogy with
\eqv{} \sjy{}
of actions of compact groups on finite dimensional \ca{s}
(Theorem~2.6 of \cite{Phi2012})
and
\pjy{} of $C_0 ( (0, 1] )$,
it seems reasonable to hope that the orthogonal
Bernoulli shift based on $C_0 ( (0, 1] )$
is projective whenever $G$ is finite.
This seems difficult to prove;
we have been able to do so only for $G = \Z_{2^n}$,
the finite cyclic group of order $2^n$.
(See \autoref{P_OrthShift}.)
We start with some lemmas.


\begin{lem}
\label{P_skewAct}
Let $n \in \N$.
Let $B$ be a \ca,
let $\bt \in \Aut (B)$ satisfy $\bt^{2^n} = \id_B$,
let $I$ be a $\bt$-invariant ideal in $B$,
let $\pi \colon B \to B / I$ be the quotient map,
and let $\af \in \Aut (B / I)$ be the induced automorphism.
If $x \in B / I$ is selfadjoint and satisfies $\af (x) = - x$,
then there is a selfadjoint element $y \in B$
such that $\bt (y) = - y$ and $\pi (y) = x$.
\end{lem}


\begin{proof}
First lift $x$ to a selfadjoint element $b \in B$.
Put
\[
y = \frac{1}{2^n} \big[ b - \bt (b)
       + \bt^{2} (b) - \bt^{3} (b) + \cdots
       - \bt^{2^{n} - 1} (b) \big].
\]
Then $y$ is selfadjoint,
$\bt (y) = - y$, and $y$ is a lift of $x$.
\end{proof}


\begin{lem}
\label{P_ShiftTwo}
Let $n \in \N$.
Let $B$ be a \ca,
let $\bt \in \Aut (B)$ satisfy $\bt^{2^n} = \id_B$,
let $I$ be a $\bt$-invariant ideal in $B$,
let $\pi \colon B \to B / I$ be the quotient map,
and let $\af \in \Aut (B / I)$ be the induced automorphism.
Let $h_1, h_2 \in B/I$ be positive orthogonal elements
such that $\alpha (h_1) = h_2$ and $\alpha (h_2) = h_1$.
Then there exist positive orthogonal elements $k_1, k_2 \in B$
such that
\[
\pi (k_1) = h_1,
\,\,\,\,\,\,
\pi (k_2) = h_2,
\,\,\,\,\,\,
\bt (k_1) = k_2,
\andeqn
\bt (k_2) = k_1.
\]
\end{lem}


\begin{proof}
Put $x = h_1 - h_2$.
Since $x$ is selfadjoint and $\alpha (x) = - x$, we can,
by \autoref{P_skewAct}, lift it to a selfadjoint
element $y \in B$ with $\bt (y) = - y$.
Let $k_1$ be the positive part of $y$, that is,
$k_1 = \frac{1}{2} (y + |y|)$.
Then $\pi (k_1) = \frac{1}{2} (x + |x|) = h_1$.
Put $k_2 = \bt (k_1)$.
Routine calculations show that
$k_2$ is the negative part of~$y$,
and thus orthogonal to $k_1$.
Essentially the same calculations show
that $\bt (k_2) = k_1$.
It is clear that $\pi (k_2) = h_2$.
\end{proof}


\begin{prp}\label{P_2X08LMnyOrth}
Let $n \in \N$.
Let $B$ be a \ca,
let $\bt \in \Aut (B)$ satisfy $\bt^{2^n} = \id_B$,
let $I$ be a $\bt$-invariant ideal in $B$,
let $\pi \colon B \to B / I$ be the quotient map,
and let $\af \in \Aut (B / I)$ be the induced automorphism.
Let $h_1, h_2, \ldots, h_{2^n} \in B/I$ be orthogonal positive elements
such that $\af (h_m) = h_{m + 1}$ for $m = 1, 2, \ldots, 2^n - 1$,
and such that $\af (h_{2^n}) = h_1$.
Then they can be lifted to orthogonal positive elements
$k_m \in B$ for $m = 1, 2, \ldots, 2^n$
such that $\bt (k_m) = k_{m + 1}$
for $m = 1, 2, \ldots, 2^n - 1$ and such that $\bt (k_{2^n}) = k_1$.
Moreover,
if $\| h_m \| \leq 1$ for $m = 1, 2, \ldots, 2^n$,
then we can require that $\| k_m \| \leq 1$ for $m = 1, 2, \ldots, 2^n$.
\end{prp}


\begin{proof}
The proof (except for the last statement) is by induction on~$n$.
The case $n = 1$ is \autoref{P_ShiftTwo}.
Let $n > 1$, suppose that we have shown
the statement to hold for all natural numbers $l < n$
and all choices of $B$, $\bt$, $I$, and $h_1, h_2, \ldots, h_{2^l}$,
and let $B$, $\bt$, $I$, and $h_1, h_2, \ldots, h_{2^n}$
be as in the statement.

Set
\[
a_1 = h_1 + h_3 + \cdots + h_{2^{n} - 1}
\andeqn
a_2 = h_2 + h_4 + \cdots + h_{2^n}.
\]
Then $\af (a_1) = a_2$, $\af (a_2) = a_1$, and $a_1 a_2 = 0$.
So, by \autoref{P_ShiftTwo},
we can lift $a_1$ and $a_2$ to orthogonal positive elements
$b_1, b_2$ such that $\bt (b_1) = b_2$ and $\bt (b_2) = b_1$.

The hereditary subalgebra ${\overline{b_1 B b_1}} \subset B$
is $\bt^2$-invariant
and it is easy to check that
$\pi \big( {\overline{b_1 B b_1}} \big) = {\overline{a_1 (B/J) a_1}}$.
Apply the induction hypothesis
with ${\overline{b_1 B b_1}}$ in place of~$B$,
with $\bt^2$ in place of~$\bt$,
with ${\overline{b_1 B b_1}} \cap I$ in place of~$I$,
and with $h_1, h_3, \ldots, h_{2^{n} - 1}$
in place of $h_1, h_2, \ldots, h_{2^n}$.
We obtain orthogonal positive elements
$k_1, k_3, \ldots, k_{2^{n} - 1} \in {\overline{b_1 B b_1}}$
such that $\pi (k_m) = h_m$ and $\bt^2 (k_m) = k_{m + 2}$
for $m = 1, 3 ,5, \ldots, 2^{n} - 1$,
and such that $\bt^2 (k_{2^{n} - 1}) = k_1$.
Set $k_m = \bt (k_{m - 1}) \in {\overline{b_2 B b_2}}$
for $m = 2, 4, 6, \ldots, 2^n$.
Then $\pi (k_m) = h_m$ also for $m = 2, 4, 6, \ldots, 2^n$.
It is clear that $\bt (k_m) = k_{m + 1}$
for $m = 1, 2, \ldots, 2^n - 1$ and that $\bt (k_{2^n}) = k_1$.
It only remains to check that the elements $k_m$ are orthogonal
for $m = 1, 2, \ldots, 2^n$.
The only case needing work is $k_l$ and $k_m$ when
one of $l$ and $m$ is even and the other is odd.
But then one of $k_l$ and $k_m$ is in
${\overline{b_1 B b_1}}$ and the other is in ${\overline{b_2 B b_2}}$,
so the desired conclusion follows from $b_1 b_2 = 0$.

It remains to prove the last statement.
Let $x_1, x_2, \ldots, x_{2^n} \in B$
be the elements produced in the first part.
Let $f \colon [0, \infty) \to [0, 1]$
be the function $f (t) = \min (t, 1)$ for $t \geq 0$.
Then set $k_m = f (x_m)$ for $m = 1, 2, \ldots, 2^n$.
\end{proof}


\begin{prp}
\label{P_OrthShift}
Let $A$ be a separable \sjpj{} \ca,
let $n \in \N$,
and let
$\sm \colon \Z_n \to \Aut \big( \bigoplus_{m = 1}^n A \big)$
be the orthogonal Bernoulli shift of \autoref{R:OrthShift}.
If $n$ is a power of $2$, then $\sm$ is \sjpj.
\end{prp}


\begin{proof}
We give the proof when $A$ is projective.
The semiprojective case is analogous,
but requires \autoref{L_2X05IdAlgInt}.

By \autoref{P_SjInjSurj}, it is enough
to show that for every $G$-algebra $\GBb$
and every surjective $G$-morphism
$\pi \colon B \to \bigoplus_{m = 1}^n A$,
there exists a $G$-morphism $\ps \colon \bigoplus_{m = 1}^n A \to B$
such that $\pi \circ \ps = \id_A$.

Let $\af = \sm_1$, the automorphism corresponding to the generator
$1 \in \Z_n$, and similarly let $\gm = \bt_1 \in \Aut(B)$.
For $m = 1, 2, \ldots, n$, let
$\io_m \colon A \to \bigoplus_{m = 1}^n A$
be the map that sends $A$ to the summand
in $\bigoplus_{m = 1}^n A$ indexed by~$m$,
and let $\rh_m \colon \bigoplus_{m = 1}^n A \to A$
be the surjection onto the summand indexed by~$m$.
Then $a = \sum_{m = 1}^n (\io_m \circ \rh_m) (a)$
for every $a \in \bigoplus_{m = 1}^n A$.

Let $h$ be a strictly positive element in $A$.
For $m = 1, 2, \ldots, n$, set $h_m = \io_m (h)$.
Then $h_1, h_2, \ldots, h_n$ are orthogonal positive elements
in $\bigoplus_{m = 1}^n A$
such that $\af (h_m) = h_{m + 1}$ for $m = 1, 2, \ldots, n - 1$,
and such that $\af (h_n) = h_1$.
By \autoref{P_2X08LMnyOrth},
they can be lifted to orthogonal positive elements
$k_m \in B$ for $m = 1, 2, \ldots, 2^n$
such that $\gm (k_m) = k_{m + 1}$
for $m = 1, 2, \ldots, n - 1$ and such that $\gm (k_n) = k_1$.

Let $D = {\overline{k_1 B k_1}}$.
Then $\pi (D) = \io_1 (A)$.
Since $A$ is projective,
there exists a \stHm{} $\ps_1 \colon A \to D$
such that $\pi \circ \ps_1 = \io_1$.
Define $\ps \colon \bigoplus_{m = 1}^n A \to B$ by
\[
\ps(a) = \sum_{m = 1}^n
       \big( \gm^{m - 1} \circ \ps_1 \circ \rho_m \big) (a).
\]
It is easily checked that $\ps$ has the desired properties.
\end{proof}


\begin{qst}
Consider the orthogonal Bernoulli shift
$G \to \Aut \big( \bigoplus_G C_0 ( (0, 1] ) \big)$
of \autoref{R:OrthShift}.
For which groups~$G$ is this action \pj?

In particular, is it \pj{} for $G = \Z_n$
for all $n \in \N$?
Is it \pj{} for $G = \Z$?
\end{qst}


\section{Equivariant semiprojectivity of restrictions to subgroups}
\label{Sec_SjSubGp}

\indent
Let $\af \colon G \to \Aut (A)$ be a \sj{} action.
In this section we investigate \sjy{}
of the restriction
of $\af$ to a subgroup $H \leqslant G$.
In \autoref{P_SjSubGp},
we obtain a positive result when $H$ is cocompact.
It follows (\autoref{P_EqSj_implies_Sj})
that a second countable compact group
can only act \sj{ly} on a \ca{} that is \sj{} in the usual sense.
Some condition on~$H$ is necessary.
For instance,
in \autoref{R:SjSubGp_notCpct}
we construct a \sj{} action of $\Z$ on a non\sj{} \ca.

To obtain these results, we use the induction functor,
which assigns in a natural way to each $H$-algebra an induced
$G$-algebra.
We will show that this functor
preserves exact sequences (\autoref{P_Ind_Exact})
and behaves well with respect to direct limits
(\autoref{P_Ind_Cts}).

We begin by recalling the definition of the induction functor,
from the beginning of Section~2 of \cite{KIRWAS1999B}
or the beginning of Section~6 of \cite{Ech2010}.
In \autoref{D_2X08Ind} below,
one easily checks that the action defined on the algebra $\Ind_H^G (A)$
is \ct,
so that $\big( G, \, \Ind_H^G (A), \, \Ind_H^G (\af) \big)$
is in fact a $G$-algebra,
and that $\Ind_H^G$ really is a functor.


\begin{dfn}\label{D_3205_Cat}
For a \lcg{} $G$, we let $\CatG$ denote
the category whose objects are $G$-algebras
and whose morphisms are $G$-equivariant \stHm{s}
(also called $G$-morphisms).
\end{dfn}


\begin{dfn}\label{D_2X08Ind}
Let $H \leqslant G$ be a closed subgroup,
and let $(H, A, \af)$ be an object in $\CatH$.
We define an object $\big( G, \, \Ind_H^G (A), \, \Ind_H^G (\af) \big)$
in $\CatG$ as follows.
We take
\[
\Ind_H^G (A)
 = \left\{ f \in C_{\mathrm{b}} (G, A) \colon
\begin{matrix}
{\mbox{\; $\af_h (f (s h)) = f (s)$ for all $s \in G$ and $h \in H$}}
 \\
{\mbox{and $s H \mapsto \| f (s) \|$ is in $C_0 (G / H)$}}
\end{matrix}
\right \}.
\]
The induced action
$\Ind_H^G (\af) \colon G \to \Aut \big( \Ind_H^G (A) \big)$
is given by
\[
\big( \Ind_H^G (\af) \big)_s (f) (t) = f ( s^{- 1} t)
\]
for $f \in \Ind_H^G (A)$ and $s, t \in G$.
If $A$ and $B$ are $H$-algebras and
$\ph \colon A \to B$ is an $H$-morphism,
then the induced $G$-morphism
$\Ind_H^G (\ph) \colon \Ind_H^G (A) \to \Ind_H^G (B)$ is given by
\[
\Ind_H^G (\ph) (f) (s) = \ph (f (s))
\]
for $f \in \Ind_H^G (A)$ and $s \in G$.
\end{dfn}


The induction functor is often defined on a different category than
that considered here.
The objects are still $G$-algebras, but the morphisms
are equivariant Hilbert bimodules.
We refer to Section~6 of \cite{Ech2010}
and to \cite{EchKalQuiRae2000}
for more details.

We next recall the definition of a $C_0 (X)$-algebra.
See Section~4.5 of \cite{Phi1987},
Definition~1.5 of \cite{Kas1988},
or Definition~2.6 of \cite{Blanch1996}.
We recall that if $A$ is a \ca,
then $M (A)$ is its multiplier algebra
and $Z (A)$ is its center.


\begin{dfn}\label{D_2X08c0XAlg}
Let $X$ be a locally compact Hausdorff space.
A {\emph{$C_0 (X)$-algebra}} is a \ca~$A$ together with a \stHm{}
$\eta \colon C_0 (X) \to Z (M (A))$, called the structure map,
such that
\[
\big\{ \et (f) a \colon
     {\mbox{$f \in C_0 (X)$ and $a \in A$}} \big\}
\]
is dense in~$A$.

We will usually write $f a$ or $f \cdot a$ instead of $\eta (f) a$
for the product of a function $f \in C_0 (X)$ and an element $a \in A$.
For an open set $U \subset X$, we set
\[
A (U) = \big\{ f a \colon
     {\mbox{$f \in C_0 (U)$ and $a \in A$}} \big\},
\]
which is an ideal of $A$.
(See \autoref{P_2X08PropCXA}(\ref{P_2X08PropCXA_Ideal}).)
For a closed subset $Y \subset X$, we denote by $A (Y)$
the quotient $A / A (X \setminus Y)$.

For $x \in X$ we write $A (x)$ for $A (\{x \})$,
and this \ca{} is called the {\emph{fiber of $A$ at $x$}}.
Given $a \in A$, we denote its image in the fiber $A (x)$ by $a (x)$,
and we define $\check{a} \colon X \to [0, \infty)$
by $\check{a} (x) = \| a (x) \|$ for $x \in X$.
We call $A$ a {\emph{continuous}} $C_{0} (X)$-algebra
if $\check{a}$ is continuous for each $a \in A$.

If $A$ and $B$ are $C_0 (X)$-algebras
and $\ph \colon A \to B$ is a \stHm,
then $\ph$ is said to be a {\emph{$C_0 (X)$-morphism}}
if $\ph (f \cdot a) = f \cdot \ph (a)$
for all $f \in C_0 (X)$ and $a \in A$.
\end{dfn}


We recall the following facts about $C_0 (X)$-algebras.


\begin{prp}\label{P_2X08PropCXA}
Let $X$ be a locally compact Hausdorff space
and let $A$ be a $C_0 (X)$-algebra
with structure map $\eta \colon C_0 (X) \to Z (M (A))$.
Then:
\begin{enumerate}
\item\label{P_2X08PropCXA_CFact}
$A = \big\{ \et (f) a \colon
     {\mbox{$f \in C_0 (X)$ and $a \in A$}} \big\}$.
\item\label{P_2X08PropCXA_Ideal}
If $U \subset A$ is open then $A (U)$ is an ideal in~$A$.
\item\label{P_2X08PropCXA_USC}
For $a \in A$,
the function $\check{a}$ is upper semicontinuous
and vanishes at infinity.
\item\label{P_2X08PropCXA_Norm}
For $a \in A$,
we have $\| a \| = \sup_{x \in X} \check{a} (x)$.
\end{enumerate}
\end{prp}


\begin{proof}
Part~(\ref{P_2X08PropCXA_CFact})
is Proposition~1.8 of \cite{Blanch1996}.
(This is essentially the Cohen Factorization Theorem.)

For~(\ref{P_2X08PropCXA_Ideal}),
it follows from Corollary~1.9 of \cite{Blanch1996}
that $A (U)$ is a closed $C_0 (X)$-submodule of~$A$.
It now easily follows that $A (U)$ is an ideal.

Part~(\ref{P_2X08PropCXA_USC})
is Proposition~1.2 of \cite{Rie1989}.

Part~(\ref{P_2X08PropCXA_Norm})
is Proposition~2.8 of \cite{Blanch1996}.
\end{proof}

We refer to Section~2 of \cite{Blanch1996}
for more details on $C_0 (X)$-algebras.


\begin{prp}\label{P_2X08_IndIsCXA}
Let $G$ be a \lcg, let $H \leqslant G$ be a closed subgroup,
and let $(H, A, \af)$ be an $H$-algebra.
Define
$\eta \colon C_0 (G / H) \to Z \big( M \big( \Ind_H^G (A) \big) \big)$
by
\[
(\eta (g) f) (s) = g(s H) \cdot f (s)
\]
for $g \in C_0 (G / H)$, $f \in \Ind_H^G (A)$, and $s \in G$.
This map makes $\Ind_H^G (A)$
a continuous $C_0 (G / H)$-algebra.
Moreover:
\begin{enumerate}
\item\label{P_2X08_IndIsCXA_Mor}
If $(H, B, \bt)$ is a second $H$-algebra,
and $\ph \colon A \to B$ is an $H$-morphism,
then $\Ind_H^G (\ph)$ is a morphism of $C_0 (G / H)$-algebras.
\item\label{P_2X08_IndIsCXA_Fib}
For every $x \in G$,
the map $\ev_x \colon \Ind_H^G (A) \to A$,
which evaluates a function in $\Ind_H^G (A)$ at $x$,
defines an isomorphism from $\Ind_H^G (A) (x H)$ to~$A$.
\end{enumerate}
\end{prp}


In particular, the fibers of $\Ind_H^G (A)$
as a $C_0 (G / H)$-algebra are all isomorphic to~$A$.
However, the isomorphism is not canonical.
In the proof below, the isomorphism for the fiber at $x H \in G / H$
depends on the choice of the coset representative~$x$.


\begin{proof}[Proof of \autoref{P_2X08_IndIsCXA}]
It is easy to check that $\et$ makes $\Ind_H^G (A)$
a continuous $C_0 (G / H)$-algebra,
and we omit the details.
The proof of~(\ref{P_2X08_IndIsCXA_Mor}) is immediate.
It remains to prove~(\ref{P_2X08_IndIsCXA_Fib}).
We abbreviate $\Ind_H^G (A)$ to $\Ind (A)$.

Let $x \in G$.
We show that $\ev_x$ is surjective.
It is immediate that
\[
\ker ( \ev_x) = \Ind (A) ( G / H \setminus \{ x H \} ),
\]
so this will complete the proof.

Since $\ev_x$ is a \stHm,
it is enough to show that it has dense image in $A$.
So let $a \in A$ and let $\ep > 0$.
We want to find $f \in \Ind_H^G (A)$ such that $\| f (x) - a \| < \ep$.
Let $\mu$ denote the Haar measure of $H$.
Since the action is continuous, there exists
an open neighborhood $U \subset H$
of the identity element $1 \in H$,
with compact closure,
such that $\| \af_s (a) - a \| \leq \frac{\ep}{2}$ for all $s \in U$.
Let $\ch \colon G \to [0, \infty)$ be a nonzero \ct{} function
with $\supp (\ch) \subset U$.
By scaling, we may assume $\int_H \ch \, d \mu = 1$.
We define a function $f \colon G \to A$ by
\[
f (s) = \int_H \ch (x^{- 1} s t) \cdot \af_t (a) \, d \mu (t)
\]
for $s \in G$.
The integral exists for all $s$, since the integrand
is continuous and has compact support.
We will now check that $f$ has the desired properties.

For $s \in G$ and $h \in H$ we have,
using left invariance of $\mu$ at the last step,
\[
\af_h ( f (s h) )
 = \af_h \left(
      \int_H \ch (x^{- 1} s h t) \cdot \af_t (a) \, d \mu (t)
     \right)
 = \int_H \ch (x^{- 1} s h t) \cdot \af_{h t} (a) \, d \mu (t)
 = f (s).
\]

The function $s H \mapsto \| f (s) \|$ has compact support,
so $f \in \Ind (A)$.
Moreover,
\begin{align*}
\| f (x) - a \|
& = \left\| \int_H \ch (t) \cdot \af_t (a) \, d \mu (t)
         - \int_H \ch (t) \cdot a \, d \mu (t) \right\|
      \\
& \leq \int_H \ch (t) \| \af_t (a) - a \| \, d \mu (t)
  \leq \frac{\ep}{2}
  < \ep.
\end{align*}
This completes the proof that $\ev_x$ is surjective.
\end{proof}


The following result is similar to
Lemma~3.2 of \cite{ThiWin2012}.
It is Lemma 2.1(iii) of \cite{Dad2009},
but the proof given there assumes that $X$ is compact.


\begin{lem}
\label{P_StoneWei}
Let $A$ be a $C_{0} (X)$-algebra with structure map
$\eta \colon C_0 (X) \to Z (M (A))$.
Assume $B \subset A$ is a $C^*$-subalgebra
satisfying the following two conditions:
\begin{enumerate}
\item\label{P_StoneWei_Fib}
For each $x \in X$,
the set $\{ b (x) \colon b \in B \}$ exhausts the fiber $A (x)$.
\item\label{P_StoneWei_Mult}
$\eta (C_0 (X)) B \subset B$, that is,
$B$ is invariant under multiplication by functions in $C_0 (X)$.
\end{enumerate}
Then $A = B$.
\end{lem}


\begin{proof}
It suffices to show that $B$ is dense in~$A$.
Let $a \in A$ and let $\ep > 0$.
Using \autoref{P_2X08PropCXA}(\ref{P_2X08PropCXA_CFact}),
choose $f \in C_0 (X)$ and $a_0 \in A$ such that $f a_0 = a$.
Choose $g \in C_{\mathrm{c}} (X)$
such that $\| f - g \| < \ep / (2 \| a_0 \| )$.
Then $\| g a_0 - a \| < \frac{\ep}{2}$
and $(g a_0) (x) = 0$ for $x \in X \setminus \supp (g)$.

For each point $x \in \supp (g)$,
choose $b_x \in B$ such that $b_x (x) = (g a_0) (x)$.
By
\autoref{P_2X08PropCXA}(\ref{P_2X08PropCXA_USC}),
there is an open set $U_x \subset X$
with $x \in U_x$ such that for all $y \in U_x$
we have $\| b_x (y) - (g a_0) (y) \| < \frac{\ep}{2}$.
Choose $x_1, x_2, \ldots, x_n \in \supp (g)$
such that the sets $U_{x_1}, U_{x_2}, \ldots, U_{x_n}$
cover $\supp (g)$.
Choose $h_1, h_2, \ldots, h_n \in C_{\mathrm{c}} (X)$
such that for $k = 1, 2, \ldots, n$
we have
$\supp (h_k) \subset U_{x_k}$
and $0 \leq h_k \leq 1$,
and such that $\sum_{k = 1}^n h_k \leq 1$
and is equal to~$1$ on $\supp (g)$.
Set $b = \sum_{k = 1}^n h_k b_{x_k}$.
Then $b \in B$.
We claim that $\| b - g a_0 \| \leq \frac{\ep}{2}$.
This will imply that $\| b - a \| < \ep$,
and complete the proof.

It suffices to show that $\| b (y) - (g a_0) (y) \| \leq \frac{\ep}{2}$
for $y \in X$.
Set $h_0 = 1 - \sum_{k = 1}^n h_k$.
Then $g a_0 = \sum_{k = 0}^n h_k g a_0$.
Set $b_k = b_{x_k}$ and $U_k = U_{x_k}$ for $k = 1, 2, \ldots, n$,
and set $b_0 = 0$ and $U_0 = X \setminus \supp (g)$.
Then for $k = 0, 1, \ldots, n$,
we have $\| b_k (y) - (g a_0) (y) \| < \frac{\ep}{2}$
whenever $h_k (y) \neq 0$.
Using this fact at the second step, we have
\[
\| b (y) - (g a_0) (y) \|
 \leq \sum_{k = 0}^n h_k (y) \| b_k (y) - (g a_0) (y) \|
 \leq \sum_{k = 0}^n h_k (y) \cdot \frac{\ep}{2}
 \leq \frac{\ep}{2}.
\]
This proves the claim, and completes the proof.
\end{proof}


The following result is Lemma~3.8 of \cite{KIRWAS1999B},
but the proof given in \cite{KIRWAS1999B}
does not address surjectivity of $\Ind_H^G (\pi)$.


\begin{prp}
\label{P_Ind_Exact}
Let $G$ be a \lcg, and let $H \leqslant G$ be a closed subgroup.
Then the induction functor $\Ind_H^G \colon \CatH \to \CatG$ is exact,
that is, given an $H$-equivariant short exact sequence of $H$-algebras
\[
0 \longrightarrow I
  \stackrel{\io}{\longrightarrow} A
  \stackrel{\pi}{\longrightarrow} B
  \longrightarrow 0,
\]
the induced $G$-equivariant sequence of $G$-algebras
\[
\xymatrix{
0 \ar[r]
& \Ind_H^G (I) \ar[rr]^{\Ind_H^G (\io)}
&
& \Ind_H^G (A) \ar[rr]^{\Ind_H^G (\pi)}
&
& \Ind_H^G (B) \ar[r]
& 0
}
\]
is also exact.
\end{prp}


\begin{proof}
To simplify the notation, we abbreviate $\Ind_H^G$ to $\Ind$.

We may think of $I$ as an $H$-invariant ideal in $A$,
so that $\io$ is just the inclusion.
It follows that $\Ind (I)$ may be considered
as an ideal in $\Ind (A)$,
and then $\Ind (\io)$ is also just the inclusion morphism.

It is straightforward to check
that the sequence is exact in the middle,
that is, $\ker(\Ind (\pi)) = \Ind (I) \subset \Ind (A)$.
Thus, it remains to check that $\Ind (\pi)$ is surjective.
Following \autoref{P_2X08_IndIsCXA},
we consider $\Ind (A)$ and $\Ind (B)$ as $C_0 (G / H)$-algebras.
We want to apply \autoref{P_StoneWei}.

Condition~(\ref{P_StoneWei_Mult}) of \autoref{P_StoneWei}
follows immediately from
\autoref{P_2X08_IndIsCXA}(\ref{P_2X08_IndIsCXA_Mor}).

Let us verify condition~(\ref{P_StoneWei_Fib}).
For $x \in G$, let $\ev_x^A \colon \Ind (A) \to A$ and
$\ev_x^B \colon \Ind (B) \to B$ be the evaluation maps at $x$.
By \autoref{P_2X08_IndIsCXA}(\ref{P_2X08_IndIsCXA_Fib}),
these maps are surjective and implement the isomorphisms
$\Ind (A) (x H) \cong A$ and $\Ind (B) (x H) \cong B$.
We have $\ev_x^B \circ \Ind (\pi) = \pi \circ \ev_x^A$, that is,
the following diagram commutes:
\[
\xymatrix{
\Ind (A) \ar[d]_{\Ind (\pi)} \ar[r]^>>>>>{\ev_x^A}
& A \ar[d]^{\pi}  \\
\Ind (B) \ar[r]^>>>>>{\ev_x^B}
& B. \\
}
\]
Since $\ev_x^A$ and $\pi$ are surjective,
it follows that the image of $\Ind (\pi)$ exhausts each fiber
of $\Ind (B)$.
This verifies condition~(\ref{P_StoneWei_Fib}) of \autoref{P_StoneWei}.
So $\Ind (\pi)$ is surjective.
\end{proof}


\begin{prp}
\label{P_Ind_Cts}
Let $G$ be a \lcg, and let $H \leqslant G$ be a closed subgroup.
Then the induction functor $\Ind_H^G \colon \CatH \to \CatG$
is continuous, that is, given an $H$-equivariant direct system
\[
A_1 \longrightarrow A_2
\longrightarrow A_3
\longrightarrow \cdots,
\]
there is a natural isomorphism
\[
{\operatorname{Ind}}_H^G \big( \varinjlim A_k \big)
  \cong \varinjlim {\operatorname{Ind}}_H^G (A_k).
\]
\end{prp}


\begin{proof}
To simplify the notation, we abbreviate $\Ind_H^G$ to $\Ind$.
As explained in \autoref{P_2X08_IndIsCXA},
we consider the induced algebras as $C_0 (G / H)$-algebras.

Denote the connecting $H$-morphisms by
$\ph_m^n \colon A_m \to A_n$ for $m \leq n$.
Let $A = \varinjlim A_k$, and denote the $H$-morphisms
to the direct limit by $\ph_m^{\infty} \colon A_m \to A$.
Denote the induced $G$-morphisms by
$\te_m^n \colon \Ind (A_m) \to \Ind (A_n)$,
and let $B = \varinjlim \Ind (A_k)$,
together with $G$-morphisms
$\te_m^{\infty} \colon \Ind (A_m) \to B$.

The maps $\ph_k^{\infty}$ induce $G$-morphisms
$\Ind (\ph_k^{\infty}) \colon \Ind (A_k) \to \Ind (A)$,
and these induce a $G$-morphism $\ps$ from the direct limit~$B$
to $\Ind (A)$.
The situation is shown in the following commutative diagram:
\[
\xymatrix{
\Ind (A_1) \ar[r]^>>>>>{\te_1^2}
   \ar@/_1pc/[drrrr]_{\Ind (\ph_1^{\infty})}
& \Ind (A_2) \ar[r] \ar@/_/[drrr]^{\Ind (\ph_2^{\infty})}
& \ldots \ar[rr]
& & \varinjlim \Ind (A_k) = B \ar[d]^{\ps} \\
& & & & \Ind (A).
}
\]

To show that $\ps$ is surjective, we apply \autoref{P_StoneWei}.

To verify condition~(\ref{P_StoneWei_Mult}) of \autoref{P_StoneWei},
let $b \in B$
and $f \in C_0 (G / H)$ be given.
We will show that for every $\ep > 0$
there exists $c \in B$ such that
$\| f \cdot \ps (b) - \ps (c) \| < \ep$.
Fix $\ep > 0$.
By properties of the direct limit,
there exist $k \in \N$ and $a \in \Ind (A_k)$
such that $\| b - \te_k^{\infty} (a) \| < \ep/ \| f \|$.
One checks that $c = \te_k^{\infty} ( f \cdot a )$
has the desired properties.

To verify condition~(\ref{P_StoneWei_Fib}) of \autoref{P_StoneWei},
we need to show that every fiber of
$\Ind (A)$ is exhausted by the image of $\ps$.
We denote by $\ev_x^k \colon \Ind (A_k) \to A_k$ and
$\ev_x^{\infty} \colon \Ind (A) \to A$ the evaluation maps at $x \in G$.
Then it is enough to show that $\ev_x^{\infty} \circ \ps$ is surjective
for every $x \in G$.

For each $k \in \N$, we have
\[
\ev_x^{\infty} \circ \ps \circ \te_k^{\infty}
           = \ev_x^{\infty} \circ \Ind (\ph_k^{\infty})
           = \ph_k^{\infty} \circ \ev_x^k.
\]
Since $\ev_x^k \colon \Ind (A_k) \to A_k$ is surjective
(by \autoref{P_2X08_IndIsCXA}(\ref{P_2X08_IndIsCXA_Fib})),
the image of $\ev_x^{\infty} \circ \ps$
contains the image of $\ph_k^{\infty}$.
Thus, the image of $\ev_x^{\infty} \circ \ps$
contains $\bigcup_{k = 1}^{\infty} \im (\ph_k^{\infty})$,
which is dense in $A$ by properties of the direct limit.
It follows that the image of $\ps$ exhausts each fiber of $\Ind (A)$.
We have verified the conditions of \autoref{P_StoneWei},
so we have shown that $\ps$ is surjective.

To show that $\ps$ is injective,
let $b \in B$, and suppose that $\ps (b) = 0$.
Let $\ep > 0$;
we show that $\| b \| < \ep$.
By properties of $B$ as a direct limit,
there exist $k \in \N$ and $a \in \Ind (A_k)$
such that $\| b - \te_k^{\infty} (a) \| < \frac{\ep}{3}$.
For $n \geq k$, let $f_n \in C_0 (G / H)$ be defined by
$f_n (s H) = \| \te_k^n (a) (s H) \|$.
One checks that $(f_n)_{n \in \N}$
is a nonincreasing sequence of functions
such that $\lim_{n \to \infty} f_n (s H) < \frac{\ep}{3}$
for each $s \in G$.
For $n \in \N$, define a continuous function $g_n$ on the one point
compactification $(G / H)^{+}$ by
$g_n (s H) = \max \big( f_n (s H), \, \frac{\ep}{3} \big)$
for $s \in G$ and $g_n (\infty) = \frac{\ep}{3}$.
The functions $g_n$ decrease pointwise to the constant function
with value $\frac{\ep}{3}$.
Since $(G / H)^{+}$ is compact,
Dini's Theorem (Proposition~11 in Chapter~9 of \cite{Roy1988})
implies that the convergence is uniform.
So there exists $n \geq k$
such that $\| f_n \| < \frac{2 \ep}{3}$.
Then $\| \te_k^n (a) \| = \| f_n \| < \frac{2 \ep}{3}$
by \autoref{P_2X08PropCXA}(\ref{P_2X08PropCXA_Norm}),
and thus also $\| \te_k^{\infty} (a) \| < \frac{2 \ep}{3}$.
It follows that $\| b \| < \ep$, as desired.

This completes the proof that $\ps$ is an isomorphism.
\end{proof}


\begin{lem}
\label{P_Eval}
Let $G$ be a \lcg, and let $H \leqslant G$ be a closed subgroup.
For any $H$-algebra $A$,
let $\ev_1^A \colon \Ind_H^G (A) \to A$
be the map $\ev_1^A (f) = f (1)$ that evaluates a function at the
identity element $1 \in G$.
Then  $\ev_1^A$ is an $H$-morphism, and is natural in~$A$.
\end{lem}


\begin{proof}
We need only check equivariance.
Let $\af \colon H \to \Aut (A)$ denote the action on~$A$.
Let $\gm = \Ind_H^G (\af)$ be the induced action of $G$
on $\Ind_H^G (A)$.
For $f \in \Ind_H^G (A)$ and $h \in H$,
we have, using the definition of $\Ind_H^G (A)$
at the third step,
\[
\ev_1^A (\gm_h (f))
  = (\gm_h (f)) (1)
  = f (h^{- 1})
  = \af_h (f (1))
  = \af_h (\ev_1^A (f)),
\]
as desired.
\end{proof}


\begin{lem}
\label{P_Maps}
Let $G$ be a \lcg, and let $H \leqslant G$ be a closed subgroup
such that $G / H$ is compact.
Let $(G, A, \af)$ be a $G$-algebra,
and let $(H, B, \bt)$ be an $H$-algebra.
Let $\ph \colon A \to B$ be an $H$-morphism.
Then there is a $G$-morphism $\et \colon A \to \Ind_H^G (B)$
such that $\et (a) (s) = \ph \big( \af_s^{- 1} (a) \big)$
for all $a \in A$ and $s \in G$.
\end{lem}


\begin{proof}
We only have to prove that the formula for $\et (a)$
defines an element of $\Ind_H^G (B)$
and that the resulting map from $A$ to $\Ind_H^G (B)$
is $G$-equivariant.
Let $a \in A$.

For the first,
since $G / H$ is compact, the function
$s H \mapsto \| \eta (a) (s) \|$ is obviously in $C_0 (G / H)$.
Let $s \in G$ and $h \in H$.
Then
\[
\bt_h \big( \et (a) (s h) \big)
  = \bt_h \big( \ph \big( \af_{s h}^{- 1} (a) \big) \big)
  = \ph \big( \af_h \circ \af_{h^{- 1} s^{- 1}} (a) \big)
  = \et (a) (s),
\]
as desired.

For the second,
let $\gm = \Ind_H^G (\af)$ be the action of $G$ on $\Ind_H^G (B)$.
Let $s, t \in G$.
Then
\[
\gm_s (\et (a)) (t)
 = \et (a) (s^{- 1} t)
 = \ph \big( \af_{t^{- 1}} (\af_s (a)) \big)
 = \et (\af_s (a)) (t),
\]
as desired.
\end{proof}


\begin{thm}
\label{P_SjSubGp}
Let $G$ be a \lcg, and let $H \leqslant G$ be a closed subgroup
such that $G / H$ is compact.
Let $(G, A, \af)$ be a $G$-algebra.
If $\af$ is \eqsj,
then $\af |_H$ is \eqsj.
\end{thm}


\begin{proof}
Let $(H, C, \gm)$ be an $H$-algebra.
To simplify the notation, we abbreviate $\Ind_H^G$ to $\Ind$.
The maps to be introduced are shown
in the diagram near the end of the proof.
Let
$J_0 \subset J_1 \subset \cdots$ be $H$-invariant ideals in~$C$,
let $J = {\overline{\bigcup_{n = 0}^{\infty} J_n}}$,
let
\[
\kp \colon C \to C / J,
\, \, \, \, \, \,
\kp_n \colon C \to C / J_n,
\andeqn
\pi_n \colon C / J_n \to C / J
\]
be the quotient maps,
and let $\ph \colon A \to C / J$
be an $H$-morphism.
Then
\[
\Ind (J)
 = {\overline{\bigcup_{n = 0}^{\infty} \Ind (J_n)}}
\]
by \autoref{P_Ind_Cts}.
Moreover, \autoref{P_Ind_Exact} allows us to identify
the quotients
$\Ind (C) / \Ind (J_n)$
with $\Ind (C / J_n)$
and $\Ind (C) / \Ind (J)$
with $\Ind (C / J)$,
with quotient maps
\[
\Ind (\kp)
 \colon \Ind (C)
   \to \Ind (C) / \Ind (J),
\,\,\,\,\,\,
\Ind (\kp_n)
 \colon \Ind (C)
   \to \Ind (C) / \Ind (J_n),
\]
and
\[
\Ind (\pi_n)
 \colon \Ind (C) / \Ind (J_n)
    \to \Ind (C) / \Ind (J).
\]

Let $\et \colon A \to \Ind (C) / \Ind (J)$
be as in \autoref{P_Maps}.
Since $\af$ is \eqsj,
there exist $n \in \N$ and a $G$-morphism
$\ld \colon A \to \Ind (C) / \Ind (J_n)$
such that $\Ind (\pi_n) \circ \ld = \et$.
We now have the following commutative diagram,
with the horizontal maps on the right being as
in \autoref{P_Eval}:
\[
\xymatrix{
& & \Ind (C)
      \ar[d]^{\Ind (\kp_n)} \ar[rr]^(0.55){\ev_1^C}
   & & C \ar[d]^{\kp_n} \ar@/^2pc/[dd]^{\kp} \\
& & \Ind (C) / \Ind (J_n)
     \ar[d]^{\Ind (\pi_n)}
                \ar[rr]^(0.65){\ev_1^{C / J_n}}
   & & C / J_n \ar[d]^{\pi_n} \\
A \ar[rr]_-{\et} \ar[rru]^{\ld}
   & & \Ind (C) / \Ind (J)
      \ar[rr]^(0.65){\ev_1^{C / J}}
   & & C / J.
}
\]
It is easy to check that $\ev_1^{C / J} \circ \et = \ph$.
Therefore the map $\ps = \ev_1^{C / J_n} \circ \ld$
is an $H$-morphism from $A$ to $C / J_n$
such that $\pi_n \circ \ps = \ph$.
\end{proof}


\begin{cor}
\label{P_EqSj_implies_Sj}
Let $G$ be a compact group,
and let $A$ be a $G$-algebra that is \eqsj.
Then $A$ is (non\eqv{ly}) \sj.
\end{cor}


In \autoref{P_SjSubGp}, some condition on $G / H$ is necessary,
as the following example shows.


\begin{exa}
\label{R:SjSubGp_notCpct}
Let $A = C (S^1)$ be the universal \ca{} generated by a unitary,
and consider the free Bernoulli shift
$\ta \colon \Z \to \Aut \big( \BFP_{\Z} \ C (S^1) \big)$
of \autoref{D_2X07FBS}.
This action is \sj{} by \autoref{P_BernShift_EqSj_General},
but its restriction to the trivial subgroup is not.
\end{exa}


Thus, $\Z$ can act \sj{ly} on non\sj{} \ca{s}.
This is in contrast to the projective case,
discussed in \autoref{R:PjDiscreteSubGp}.
An analogous example can be constructed for any
infinite countable discrete group in place of~$\Z$.


\section{Equivariant projectivity of restrictions to subgroups}
\label{Sec_PjSubGp}

\indent
In this section we study the projective analog of the
question of \autoref{Sec_SjSubGp}.
Given a \pj{} action $\af \colon G \to \Aut (A)$,
we show in \autoref{P_2X14PjSubGp} that the restriction of $\af$
to a subgroup $H \leqslant G$ is also \pj{}
in considerable generality.
The condition we have to put is either a restriction on
the subgroup (namely that $H$ or $G / H$ is compact)
or that $G$ is a \SINgp, in which case $H$ can be arbitrary.
A \SINgp{} is a topological group for which
the right and left uniform structures agree.
See the paragraph before \autoref{P_LUCInd-Exact}.
We do not know if these hypotheses can be removed.

Since the trivial subgroup is compact,
it follows that every \eqpj{} \ca{}
is (non\eqv{ly}) \pj.
See \autoref{P_EqPj_implies_Pj}.
This is in contrast to the \sj{} case.
See \autoref{R:SjSubGp_notCpct}.

To obtain the results in this section,
we use a different induction functor,
which considers uniformly continuous functions; see \autoref{R_NCptInd}.
To show that this functor is exact,
we need a criterion for when uniformly continuous functions
into quotient \ca{s} can be lifted
to uniformly continuous functions.
In \autoref{P_UniformLifts},
we solve this problem in some generality,
and we think that this result might also be of independent interest.

There are several equivalent ways to define a uniform space.
We will mostly use the concept of a uniform cover to define
a uniformity on a set.
We refer to Isbell's book \cite{Isb1964}
for the theory of uniform spaces.
The basic definitions are in Chapter~I.
The definition of a uniformity is before item~6
in Chapter~I of \cite{Isb1964}.

If ${\mathcal{U}}$ and ${\mathcal{V}}$ are
covers of a space~$X$,
we write ${\mathcal{V}} \leq {\mathcal{U}}$
to mean that ${\mathcal{V}}$ refines ${\mathcal{U}}$.

\begin{dfn}\label{D_2X09UnifFMet}
Let $(X, d)$ be a metric space.
For $\ep > 0$ and $x \in X$,
define $U_{\ep} (x) = \{ y \in X \colon d (x, y) < \ep \}$.
The {\emph{basic uniform covers}} of~$X$ are the collections
\[
{\mathcal{B}} (\ep) = \{ U_{\ep} (x) \colon x \in X \}
\]
for $\ep > 0$.
A cover ${\mathcal{U}}$ of $X$ is called {\emph{uniform}}
if there exists $\ep > 0$
such that ${\mathcal{B}} (\ep) \leq {\mathcal{U}}$.
\end{dfn}

The proof of the following result is essentially contained
in items 1--3 in Chapter~I of \cite{Isb1964}.
One should note that if $(X, d)$ is a metric space,
$\ep_1, \ep_2 > 0$,
and ${\mathcal{U}_1}$ and ${\mathcal{U}}_2$ are covers of $X$
such that ${\mathcal{B}} (\ep_1) \leq {\mathcal{U}}_1$
and ${\mathcal{B}} (\ep_2) \leq {\mathcal{U}}_2$,
then ${\mathcal{B}} \big( \min (\ep_1, \ep_2) \big)$
refines both ${\mathcal{U}_1}$ and ${\mathcal{U}}_2$,
so that the collection of uniform covers in \autoref{D_2X09UnifFMet}
is downwards directed.
Uniformly continuous functions are defined
after Theorem~11 in Chapter~I of \cite{Isb1964},
and equiuniformly continuous families of functions are
defined before item~19 in Chapter III of \cite{Isb1964}.
The usual notion for functions on metric spaces
is just that a family $F$ of functions
from $(X_1, d_1)$ to $(X_2, d_2)$ is equiuniformly continuous
if for every $\ep > 0$ there is $\dt > 0$ such that
whenever $x, y \in X_1$ satisfy $d_1 (x, y) < \dt$,
then for all $f \in F$ we have $d_2 (f (x), \, f (y)) < \ep$.

\begin{prp}\label{P_2X08_MUisU}
Let $(X, d)$ be a metric space.
Then the collection of uniform covers in \autoref{D_2X09UnifFMet}
is a uniform structure on~$X$.
Moreover, for any two metric spaces $(X_1, d_1)$ and $(X_2, d_2)$,
the uniformly continuous functions
and the equiuniformly continuous families of functions
from $X_1$ to $X_2$
are the uniformly continuous functions
and the equiuniformly continuous families
as traditionally defined in terms of the metrics.
\end{prp}

The following theorem is the key result.
We warn that the term ``subordinate''
is used in \cite{Isb1964} with a meaning inconsistent
with its standard meaning
in the context of ordinary partitions of unity.

\begin{thm}[Theorem~11 in Chapter~IV of \cite{Isb1964}]
\label{T_2X09UPartU}
Let $X$ be a uniform space and let ${\mathcal{U}}$
be a uniform cover of~$X$.
Then there is an equiuniformly continuous
(but not necessarily locally finite) partition of unity
$( h_{U} )_{U \in {\mathcal{U}}}$
such that $h_U (x) = 0$
for all $U \in {\mathcal{U}}$ and $x \in X \setminus U$.
\end{thm}


We recall the following standard definition.


\begin{dfn}\label{D_2X09_OrdCov}
Let $X$ be a set,
and let ${\mathcal{U}}$ be a cover of~$X$.
The {\emph{order}} of ${\mathcal{U}}$,
denoted $\ord ({\mathcal{U}})$,
is the least number $n \in \Nz$
such that whenever $U_0, U_1, \ldots, U_n \in {\mathcal{U}}$
are distinct, then $U_0 \cap U_1 \cap \cdots \cap U_n = \varnothing$.
We take $\ord ({\mathcal{U}}) = \infty$ if no such~$n$ exists.
\end{dfn}


Equivalently,
$\ord ({\mathcal{U}})$
is the largest number $n$ such that there are $n$ distinct elements
of ${\mathcal{U}}$ which have nonempty intersection.
We warn the reader that some authors
use a different convention,
in which what we defined above is order $n - 1$.
For example, see page~111 of \cite{Pea1975}.
We are following the convention
implicitly used in our reference~\cite{SegSpiGue1993}.

The first part of the following definition
is found at the very beginning of Chapter~V of \cite{Isb1964},
where the term ``large dimension'' is used.
The second part is Definition~1.7 of \cite{SegSpiGue1993}.


\begin{dfn}\label{D_2X09_LUD}
Let $X$ be a uniform space.
Then the {\emph{large uniform dimension of $X$}},
denoted $\udim (X)$,
is the least $n \in \{ -1, 0, 1, 2, \ldots, \infty \}$
such that for every uniform open cover ${\mathcal{U}}$ of~$X$
there is a uniform open cover ${\mathcal{V}}$ of $X$
of order at most $n + 1$ which refines ${\mathcal{U}}$.
(We take $\udim (\varnothing) = - 1$.)

We say that $X$ is {\emph{uniformly finitistic}}
if for every uniform open cover ${\mathcal{U}}$ of~$X$
there is a uniform open cover ${\mathcal{V}}$ of $X$
of finite order which refines ${\mathcal{U}}$.
\end{dfn}


An equivalent condition for being uniformly finitistic
is that there exists a base for the uniformity
consisting of uniform covers of finite order.

If a uniform space $X$ is locally compact and paracompact
(in the induced topology),
then its covering dimension
is bounded by its large uniform dimension, that is,
$\dim(X) \leq \udim(X)$.
To see this,
let $\locdim(X)$ be the local covering dimension of~$X$
(Definition 5.1.1 of \cite{Pea1975}).
Proposition 5.3.4 of \cite{Pea1975}
gives $\dim(X) = \locdim(X)$.
For a locally compact Hausdorff space~$X$,
let $X^{+}$ denote the one point compactification of~$X$.
It is a standard result that $\locdim (X) = \dim(X^{+})$;
for instance, this is easily deduced
from Propositions 3.5.6, 5.2.1, 5.2.2, and 5.3.4 of \cite{Pea1975}.
It follows from Theorems V.5 and VI.2 of \cite{Isb1964}
that for every compactification $\gm X$ of~$X$
we have $\dim(\gm X) \leq \udim(X)$.
Thus, if $X$ is locally compact and paracompact,
we may combine these results to obtain
\[
\dim(X) = \locdim(X) = \dim(X^{+}) \leq \udim(X),
\]
as desired.

The concept of being finitistic was first defined for
topological spaces, where it means that every open cover
can be refined by an open cover of finite order.
This definition is implicit in \cite{Swa1959},
although the term ``finitistic''
was only later introduced by Bredon
on page~$133$ of his book \cite{Bre1972}.

In general, for a uniform space there is no connection
between being finitistic and uniformly finitistic.
Example~(d) after Definition~1.7 of \cite{SegSpiGue1993}
gives a uniformly finitistic space which is not finitistic.
Example~2.4 of \cite{Isb1959}
gives a discrete uniform space,
hence obviously finitistic,
with a uniform open cover having no uniform open refinement of
finite order,
thus not uniformly finitistic.


\begin{ntn}\label{N_2X09_CbX}
Let $X$ be a topological space and let $A$ be a \ca.
We denote by $\Cb (X, A)$ the \ca{}  of all bounded continuous
functions from $X$ to~$A$,
with the supremum norm.
If $X$ is a uniform space,
we let $\Cu (X, A) \subset \Cb (X, A)$
denote the subset
consisting of all bounded uniformly continuous
functions from $X$ to~$A$.
\end{ntn}


\begin{prp}\label{P_2X13CuCst}
Let $X$ be a uniform space and let $A$ be a \ca.
Then $\Cu (X, A)$ is a \ca.
\end{prp}


\begin{proof}
It is easy to check that $\Cu (X, A)$ is closed under
the algebraic operations.
That it is norm closed in $\Cb (X, A)$
follows from Corollary~32 in Chapter III of \cite{Isb1964}.
\end{proof}


The following theorem is in some sense a dual version of
Theorem~1 of \cite{Vid1969},
on the problem of extending uniformly continuous maps
from subspaces.
We do not know whether it is necessary that $X$ be uniformly finitistic.
Its proof serves as a simpler model for the proof
of \autoref{P_LUCInd-Exact}.


\begin{thm}
\label{P_UniformLifts}
Let $\pi \colon A \to B$ be
a surjective \stHm{} between two \ca s,
and let $X$ be a uniformly finitistic space.
Then the induced \stHm{}
$\kp \colon \Cu (X, A) \to \Cu (X, B)$ is surjective.
\end{thm}


\begin{proof}
It is enough to show that $\kp$~has dense range.

Given $b \in \Cu (X, B)$ and $\ep > 0$,
we will construct $a \in \Cu (X, A)$ such that
$\| \pi \circ a - b \| < \ep$.
We may clearly assume $b \neq 0$.
Let ${\mathcal{U}}$ be a uniform cover of~$X$
such that
whenever $U \in {\mathcal{U}}$ and $x, y \in U$,
then $\| b (x) - b (y) \| < \frac{\ep}{2}$.
Since $X$ is uniformly finitistic,
we may assume ${\mathcal{U}}$ has finite order.
Set $n = \ord ({\mathcal{U}})$.

Let $( h_{U} )_{U \in {\mathcal{U}}}$
be an equiuniformly continuous partition of unity
for ${\mathcal{U}}$ as in \autoref{T_2X09UPartU}.
Equiuniform continuity
in our situation means that
for every $\rh > 0$
there exists a uniform open cover ${\mathcal{V}}$ of~$X$
such that whenever $V \in {\mathcal{V}}$
and $x, y \in V$,
then for all $U \in {\mathcal{U}}$
we have $| h_U (x) - h_U (y) | < \rh$.

For each $U \in {\mathcal{U}}$ choose a point $x_U \in U$,
and let $a_U \in A$ be a lift of $b (x_U)$ with
$\| a_U \| = \| b (x_U) \|$.
For $x \in X$,
there are at most $n$ sets $U \in {\mathcal{U}}$
such that $x \in U$,
and $h_U (x)$ can be nonzero only for these sets.
Therefore the sum in the following
definition of a function $a \colon X \to A$ is finite
at each point:
\[
a (x) = \sum_{U \in {\mathcal{U}}} h_U (x) \cdot a_U
\]
for $x \in X$.
Since $\sum_{U \in {\mathcal{U}}} h_U (x) = 1$,
it further follows that $\| a \| \leq \| b \|$,
so that $a$ is bounded.

We claim that $a$ is uniformly continuous.
We follow an argument in the proof of Theorem~1 of \cite{Vid1969}.
Let $\rh > 0$.
We must find a uniform open cover ${\mathcal{V}}$ of~$X$
such that whenever $V \in {\mathcal{V}}$
and $x, y \in V$, we have $\| a (x) - a (y) \| < \rh$.
Set $\rh_0 = \rh / ( 2 n \| b \|)$.
Let ${\mathcal{V}}$ be a uniform open cover
which witnesses equiuniform continuity
of $( h_{U} )_{U \in {\mathcal{U}}}$
as above,
but with $\rh_0$ in place of $\rh$.
Let $V \in {\mathcal{V}}$ and let $x, y \in V$.
Set
\[
{\mathcal{U}}_0
  = \big\{ U \in {\mathcal{U}} \colon
   {\mbox{$x \in U$ or $y \in U$}} \big\}.
\]
Then $\card ( {\mathcal{U}}_0 ) \leq 2 n$.
Therefore
\begin{align*}
\| a (x) - a (y) \|
 & = \left\| \sum_{U \in {\mathcal{U}}_0} \big( h_U (x) - h_U (y) \big)
             \cdot a_U \right\| \\
 & \leq 2 n \cdot \| b \| \cdot
       \max_{U \in {\mathcal{U}}_0} | h_U (x) - h_U (y) |
   < 2 n \| b \| \rh_0
   = \rh.
\end{align*}
The claim is proved.

It remains to prove that $\| \pi \circ a - b \| < \ep$.
Let $x \in X$.
Then $\| \pi (a_U) - b (x) \| < \frac{\ep}{2}$ whenever
$h_U (x) \neq 0$.
Therefore
\[
\| (\pi \circ a) (x) - b (x) \|
  = \left\| \sum_{U \in {\mathcal{U}}}
          h_U (x) \big( \pi (a_U) - b (x) \big) \right\|
  \leq \sum_{U \in {\mathcal{U}}}
        h_U (x) \left\| \pi (a_U) - b (x) \right\|
  < \frac{\ep}{2}.
\]
So $\| \pi \circ a - b \| \leq \frac{\ep}{2} < \ep$,
as desired.
\end{proof}


\begin{rmk}\label{R_2X09_CStarSel}
The proof of \autoref{P_UniformLifts} can easily be adapted
to the case of bounded continuous maps.
More precisely, if $\pi \colon A \to B$ is
a surjective \stHm{} of \ca s,
and $X$ is a paracompact space,
then the method of proof shows that the induced \stHm{}
$\Cb (X, A) \to \Cb (X, B)$ is surjective.
This is a \ca{ic} version of the Bartle-Graves Selection Theorem,
Theorem~4 of \cite{BarGra1952},
which treats the case in which $A$ and $B$ are arbitrary Banach spaces.
The \ca ic version is much easier to prove
since the image of a \stHm{} is always closed.

Since a \ca{} is paracompact,
one may also formulate the theorem as follows.
Let $\pi \colon A \to B$ be a surjective \stHm{}
between \ca s.
Then there exists a \ct{} function
$\sm \colon B \to A$ (not necessarily linear)
such that $\pi \circ \sm = \id_B$ (that is, $\sm$ is a section),
and such that there is a constant $M$ such that
$\| \sm (a) \| \leq M \cdot \| a \|$
for all $a \in A$.
This also appears in Theorem~2 of~\cite{Lor1997b}.
To get this statement from the surjectivity
of $\Cb (X, A) \to \Cb (X, B)$,
take $X = \{ b \in B \colon \| b \| = 1 \}$,
lift the function $f (b) = b$ in $\Cb (X, B)$
to a bounded function $g \colon X \to A$,
and take $\sm (b) = \| b \| \cdot g \big( \| b \|^{-1} b \big)$
for $b \in B \setminus \{ 0 \}$ and $g (0) = 0$.
\end{rmk}


\begin{dfn}\label{D_2X10_LeftU}
Let $G$ be a \lcg, and let $H \leqslant G$ be a closed subgroup.
Let $q \colon G \to G / H$ be the quotient map.
For a nonempty open subset $U \subset G$ with $1 \in U$,
define ${\mathcal{B}}_{G, H} (U) = \{ q (U s) \colon s \in G \}$,
the open cover of $G / H$ by the
images in $G / H$ of the right translates of~$U$.
Define the {\emph{right uniformity}} on $G / H$
to consist of all open covers ${\mathcal{U}}$ of~$G / H$
such that there is a nonempty open subset $U \subset G$ with $1 \in U$
for which ${\mathcal{B}}_{G, H} (U) \leq {\mathcal{U}}$,
and call such covers the {\emph{right uniform covers}}.

We define the {\emph{left uniformity}} on $G$
and {\emph{left uniform covers}} of $G$ analogously,
using the covers by the left translates $\{ s U \colon s \in G \}$
for nonempty open subsets $U \subset G$ with $1 \in U$.
\end{dfn}

We do not define a left uniformity on $G / H$ since the
images in $G / H$ of the left uniform covers in $G$ will in general
not define a uniformity.


Taking $H = \{ 1 \}$,
we see that
the inversion map $s \mapsto s^{- 1}$
is uniformly continuous
if and only if the right and left uniformities on~$G$ agree.
However, for fixed $t \in G$,
both the left translation map $s \mapsto t s$
and the right translation map $s \mapsto s t$ are
uniformly continuous in the right uniformity
(and also in the left uniformity).

Uniform structures on topological groups are discussed
on pages 20--22 of \cite{HewRos1979},
but from the point of view of neighborhoods of the diagonal
rather than uniform open covers.

Clearly the map $q \colon G \to G / H$ is uniformly continuous
when both spaces are given the right uniformity.
In fact,
the right uniformity on $G / H$
is the quotient uniformity,
as defined before item~5 in Chapter~II of \cite{Isb1964},
of the right uniformity on~$G$.
We do not need this fact, so we omit the proof.


Let $G$ be a metrizable topological group.
Then $G$ has a left invariant metric determining its topology,
by Theorem 1.22 of \cite{MONZIP1955},
and analogously it also has a right invariant metric.
It is easy to check that the uniformity
induced by any right invariant metric
(as in \autoref{P_2X08_MUisU})
is equal to the right uniformity of \autoref{D_2X10_LeftU}.


Given a \lcg{} $G$ and a closed subgroup $H$,
it is shown in Lemma~2 of \cite{Ank1989}
that every (left) uniform cover of $H \setminus G$
can be refined by a cover of finite order.
In the following result we adapt the proof
to ensure that the refining cover is uniform.
We formulate the result for the space of right cosets.


\begin{prp}\label{P_CosetSp_LeftU}
Let $G$ be a \lcg, and let $H \leqslant G$ be a closed subgroup.
Then $G / H$ is uniformly finitistic
(with respect to the right uniformity).
\end{prp}

\begin{proof}
Let $\mu$ be a right Haar measure on~$G$.
We will use the notation from \autoref{D_2X10_LeftU}.
In particular, we let $q \colon G \to G / H$ be the quotient map.
The basic uniform covers of $G / H$ are defined to be
${\mathcal{B}}_{G, H} (U) = \{ q (U s) \colon s \in G \}$
for open neighborhoods $U$ of $1 \in G$.
Given such a set $U$,
we will construct a uniform cover $\mathcal{W}$ of $G / H$
that refines ${\mathcal{B}}_{G, H} (U)$
and that has finite order.

First, without loss of generality,
we may assume that $U$ has compact closure in $G$
and that $U = U^{-1}$.
Let $W$ be an open neighborhood of $1 \in G$
such that $W^3 \subset U$ and such that $W = W^{-1}$.

As in the proof of Lemma~2 in \cite{Ank1989},
we let $X$ be a maximal subset of $G$
such that the sets $q(Wx)$ for $x \in X$ are pairwise disjoint.
Set $\mathcal{W} = \{ q(W^3x) \colon x \in X \}$.
We show that it has the desired properties.

We claim that
$\mathcal{W}$ is refined by ${\mathcal{B}}_{G, H} (W)$.
To prove the claim let $g \in G$ be given.
By maximality of $X$, there exists $x \in X$
such that $q (Wg)$ and $q (Wx)$ are not disjoint.
Thus, there are $w_1, w_2 \in W$ and $h_1, h_2 \in H$
such that $w_1 g h_1 = w_2 x h_2$.
Then $g = w_1^{-1} w_2 x h_2 h_1^{-1}$
and so $q (g) \in q ( W^2 x )$.
Therefore $q(Wg) \subseteq q(W^3 x)$.
This proves the claim.

Hence, $\mathcal{W}$ is uniform and it clearly refines
the given cover ${\mathcal{B}}_{G, H} (U)$.

It remains to show that $\mathcal{W}$ has finite order.
Let $x_0, x_1, \ldots, x_k \in X$ be elements such that
$\bigcap_{j = 0}^k q ( W^3 x_j ) \neq \varnothing$.
This means that there are elements
$w_0, w_1 \ldots, w_k \in W^3$ and $h_0, h_1 \ldots, h_k \in H$
such that $w_j x_j h_j = w_0 x_0 h_0$ for $j = 1, 2, \ldots, k$.
It follows that
\[
W x_j h_j h_0^{-1} = W w_j^{-1} w_0 x_0 \subset W^3 x_0,
\]
for $j = 1, 2, \ldots, k$.

However, by construction of $X$,
the sets $W x_j h_j h_0^{-1}$ for $j = 1, 2, \ldots, k$
are pairwise disjoint.
So
\[
k \cdot \mu (W)
= \mu \left( \bigcup_{j=1}^k W x_j h_j h_0^{-1} \right)
\leq \mu \left( W^3 x_0 \right)
= \mu (W^3).
\]
Since $W$ is open and has compact closure,
$\mu (W)$ is non-zero and finite.
Thus $k \leq \mu (W^3) / \mu (W)$
and so $\mathcal{W}$ has finite order.
\end{proof}


\begin{cor}\label{P_Gp_LeftU}
Every locally compact group is uniformly finitistic
(with respect to both the right and left uniformity).
\end{cor}


\begin{ntn}\label{N_2X10_Clu}
Let $G$ be a topological group
and let $A$ be a \ca.
We denote by $\Cru (G, A)$ the \ca{} of bounded functions
$f \colon G \to A$ which are right uniformly continuous.
This is just $\Cu (G, A)$ as in Notation~\ref{N_2X09_CbX}
when $G$ is equipped with the right uniformity.
We further let $\ld \colon G \to \Aut ( \Cb (G, A) )$
be the (not necessarily continuous) action given by
$\ld_s (f) (t) = f (s^{-1} t)$ for $f \in \Cb (G, A)$
and $s, t \in G$.
\end{ntn}


Left translation is continuous on the right uniformly continuous
functions,
not the left uniformly continuous functions.
The proof is known and not difficult;
we give it here primarily to convince the reader
that the statement is correct.
We start with a preparatory lemma,
which we also need for the left uniformity.


\begin{lem}\label{L_2X10_CharRU}
Adopt Notation~\ref{N_2X10_Clu}.
Let $f \in \Cb (G, A)$.
Then $f \in \Cru (G, A)$
if and only if for every $\ep > 0$ there is an open set $V \subset G$
with $1 \in V$ such that whenever $s, t \in G$ satisfy $s t^{-1} \in V$,
then $\| f (s) - f (t) \| < \ep$.
Also,
$f$ is left uniformly continuous
if and only if for every $\ep > 0$ there is an open set $V \subset G$
with $1 \in V$ such that whenever $s, t \in G$ satisfy $t^{-1} s \in V$,
then $\| f (s) - f (t) \| < \ep$.
\end{lem}


\begin{proof}
The proofs of the two statements are the same,
and we do only the first.

First assume $f$ is right uniformly continuous.
Then there is a nonempty open set $V \subset G$ with $1 \in V$
such that whenever $s, t, g \in G$ satisfy $s, t \in V g$,
then $\| f (s) - f (t) \| < \ep$.
If now $s, t \in G$ satisfy $s t^{-1} \in V$,
then $s \in V t$ and, since $1 \in V$, also $t \in V t$.
Taking $g = t$ above,
we get $\| f (s) - f (t) \| < \ep$.

Now assume that $f$ satisfies the condition of the lemma.
Let $\ep > 0$,
and choose $V \subset G$ as in this condition.
Choose an open subset $U \subset G$ such that $1 \in U$
and $s, t \in U$ implies $s t^{-1} \in V$.
Let $s, t, g \in G$ satisfy $s, t \in U g$.
Then $s g^{-1}, \, t g^{-1} \in U$,
so $s t^{-1} = (s g^{-1}) (t g^{-1})^{-1} \in V$.
Therefore
$\| f (s) - f (t) \| < \ep$.
\end{proof}


\begin{lem}\label{L_2X10_CtUCt}
Let the notation be as in Notation~\ref{N_2X10_Clu}.
Let $f \in \Cb (G, A)$.
Then $s \mapsto \ld_s (f)$ is continuous
if and only if $f \in \Cru (G, A)$.
\end{lem}


\begin{proof}
First assume $f$ is right uniformly continuous.
Let $\ep > 0$.
It suffices to find an open subset $V \subset G$
such that $1 \in V$ and whenever $s \in V$ and $t \in G$,
then $\| \ld_{t s} (f) - \ld_t (f) \| < \ep$.
Choose an open subset $V \subset G$ as in
\autoref{L_2X10_CharRU}
with $\frac{\ep}{2}$ in place of~$\ep$.
Let $s \in V$ and $t \in G$.
Then for $g \in G$
we have
$(t^{-1} g) (s^{-1} t^{-1} g)^{-1} = s \in V$,
so
\[
\| \ld_{t s} (f) (g) - \ld_t (f) (g) \|
  = \| f (s^{-1} t^{-1} g) - f (t^{-1} g) \|
  < \frac{\ep}{2}.
\]
Taking the supremum over $g \in G$,
we get
$\| \ld_{t s} (f) - \ld_t (f) \| \leq \frac{\ep}{2} < \ep$.

For the converse, assume that
$s \mapsto \ld_s (f)$ is continuous.
We verify the criterion of \autoref{L_2X10_CharRU}.
Let $\ep > 0$.
Choose an open subset $V \subset G$
such that $1 \in V$ and whenever $s \in V$
then $\| \ld_{s} (f) - f \| < \ep$.
Let $s, t \in G$ satisfy $s t^{-1} \in V$.
Then
\[
\| f (s) - f (t) \|
  = \| f (s) - \ld_{s t^{-1}} (f) (s) \|
  \leq \| f - \ld_{s t^{-1}} (f) \|
  < \ep.
\]
This completes the proof.
\end{proof}


We now give a definition which is very similar to \autoref{D_2X08Ind},
but which uses bounded uniformly continuous functions
instead of functions vanishing at infinity.


\begin{dfn}\label{R_NCptInd}
Let $G$ be a \lcg, and let $H \leqslant G$ be a closed subgroup.
Let $\af \colon H \to \Aut (A)$ be an action of $H$ on a \ca~$A$.
We define a \ca{}  $F_H^G (A)$, with not necessarily \ct{} action
$F_H^G (\af) \colon G \to \Aut \big( F_H^G (A) \big)$,
by
\[
F_H^G (A)
  = \big\{ f \in C_{\mathrm{b}} (G, A) \colon
    {\mbox{$\af_h (f (s h)) = f (s)$ for all $s \in G$ and $h \in H$}}
        \big\}
\]
and
\[
\big( F_H^G (\af) \big)_s (f) (t)
  = f (s^{- 1} t)
\]
for $f \in F_H^G (A)$ and $s, t \in G$.
We further define a subalgebra $\LUCInd_H^G (A) \subset F_H^G (A)$
by
\[
\LUCInd_H^G (A)
 = \big\{ f \in F_H^G (A) \colon
    {\mbox{$s \mapsto \big( F_H^G (\af) \big)_s (f)$ is continuous}}
        \big\},
\]
and we take $\LUCInd_H^G (\af)$ to be the restriction of $F_H^G (\af)$
to this subalgebra.

If $A$ and $B$ are $H$-algebras and
$\ph \colon A \to B$ is an $H$-morphism,
then the induced $G$-morphisms
\[
F_H^G (\ph) \colon F_H^G (A) \to F_H^G (B)
\andeqn
\LUCInd_H^G (\ph) \colon \LUCInd_H^G (A) \to \LUCInd_H^G (B)
\]
are defined
by sending $f$ in $F_H^G (A)$ or $\LUCInd_H^G (A)$ as appropriate
to the function $s \mapsto \ph (f (s))$
for $s \in G$.

We call $\LUCInd_H^G$ the {\emph{right uniform induction functor}}.
\end{dfn}


\begin{lem}\label{L_2X10_PpOfLUCInd}
Let $G$ be a \lcg, and let $H \leqslant G$ be a closed subgroup.
Let the notation be as in \autoref{R_NCptInd}
and \autoref{D_3205_Cat}.
Then:
\begin{enumerate}
\item\label{L_2X10_PpOfLUCInd_Int}
$\LUCInd_H^G (A) = F_H^G (A) \cap \Cru (G, A)$.
\item\label{L_2X10_PpOfLUCInd_F}
$F_H^G$ is a functor from the category $\CatH$ of $H$-algebras
to the category of \ca{s} with not necessarily continuous
actions of~$G$.
\item\label{L_2X10_PpOfLUCInd_Funct}
$\LUCInd_H^G$ is a functor from $\CatH$ to~$\CatG$.
\item\label{L_2X10_PpOfLUCInd_Cpt}
If $G / H$ is compact,
then $\LUCInd_H^G = F_H^G = \Ind_H^G$.
\end{enumerate}
\end{lem}

\begin{proof}
Part~(\ref{L_2X10_PpOfLUCInd_Int}) follows from \autoref{L_2X10_CtUCt}.
Part~(\ref{L_2X10_PpOfLUCInd_F}) is an algebraic calculation.
Part~(\ref{L_2X10_PpOfLUCInd_Funct})
follows from part~(\ref{L_2X10_PpOfLUCInd_Int}),
part~(\ref{L_2X10_PpOfLUCInd_F}),
and the fact that the formula for $F_H^G (\ph)$
preserves uniform continuity.
Part~(\ref{L_2X10_PpOfLUCInd_Cpt})
follows from the observation that the condition
in \autoref{D_2X08Ind},
that $s H \mapsto \| f (s) \|$ be in $C_0 (G / H)$,
is automatic when $G / H$ is compact,
and the fact that left translation is continuous on $\Ind_H^G (A)$.
\end{proof}


To formulate the next result,
recall that a topological group $G$ is called a \SINgp\
(for ``small invariant neighborhoods'')
if every neighborhood of $1 \in G$
contains a neighborhood $V$ of $1$ that is invariant
(meaning that $g V g^{-1} = V$ for all $g \in G$).
Such groups are also called balanced.
It is easy to see that a group is a \SINgp\
if and only if the left and right uniformities on $G$ agree
(equivalently, the assignment $g \mapsto g^{-1}$
is uniformly continuous when regarded as a map
from $G$ to itself, both equipped with the right uniformity).
The class of \SINgp s includes all groups
that are abelian, compact, or discrete.
See \cite{Pal1978} for more on \SINgp{s}.


\begin{thm}
\label{P_LUCInd-Exact}
Let $G$ be a \lcg{} and let $H \leqslant G$ be a closed subgroup.
Assume that at least one of the following conditions is satisfied:
\begin{enumerate}
\item\label{P_LUCInd-Exact_Cpt}
$H$ is compact.
\item\label{P_LUCInd-Exact_CoCpt}
$G / H$ is compact.
\item\label{P_LUCInd-Exact_SIN}
$G$ is a \SINgp.
\end{enumerate}
Then the right uniform induction functor
$\LUCInd_H^G \colon \CatH \to \CatG$ is exact,
that is, given an $H$-equivariant short exact sequence of $H$-algebras
\[
0 \longrightarrow I
  \stackrel{\io}{\longrightarrow} A
  \stackrel{\pi}{\longrightarrow} B
  \longrightarrow 0,
\]
the induced $G$-equivariant sequence of $G$-algebras
\[
\xymatrix{
0 \ar[r]
& \LUCInd_H^G (I) \ar[rr]^{\LUCInd_H^G (\io)}
&
& \LUCInd_H^G (A) \ar[rr]^{\LUCInd_H^G (\pi)}
&
& \LUCInd_H^G (B) \ar[r]
& 0
}
\]
is also exact.
\end{thm}


\begin{rmk}
It seems natural to expect that \autoref{P_LUCInd-Exact}
holds in greater generality.
We do not know whether any condition is necessary
to show that the right uniform induction functor is exact.
\end{rmk}

We need two further lemmas for the proof.


\begin{lem}\label{L_2X13CcUnifC}
Let $G$ be a locally compact \SINgp,
and let $f \in C_{\mathrm{c}} (G)$.
Then for every $\ep > 0$ there is an open set $U \subset G$
such that $1 \in U$ and such that whenever $g, h, s, t \in G$
satisfy $s^{-1} t \in U$,
then $| f (g s h) - f (g t h) | < \ep$.
\end{lem}


\begin{proof}
Lemma~1.62 of~\cite{Wl} provides an open set $V \subset G$
such that $1 \in V$ and such that whenever $s, t \in G$
satisfy $s^{-1} t \in V$,
then $| f (s) - f (t) | < \ep$.
(By \autoref{L_2X10_CharRU},
this is just left uniform continuity of~$f$.)
Since $G$ is a \SINgp,
there is an invariant open set $U \subset G$
such that $1 \in U$ and $U \subset V$.
Now let $g, h, s, t \in G$
satisfy $s^{-1} t \in U$.
Then
\[
(g s h)^{-1} (g t h) = h^{-1} s^{-1} t h \in h^{-1} U h = U \subset V,
\]
so that $| f (g s h) - f (g t h) | < \ep$.
\end{proof}


\begin{lem}\label{L_2X15MeasCpt}
Let $G$ be a \lcg,
let $H \leqslant G$ be a closed subgroup,
let $\mu$ be a left Haar measure on~$H$,
and let $L \subset G$ be compact.
Then $\sup_{s \in G} \mu (s L \cap H)$ is finite.
\end{lem}


\begin{proof}
Let $q \colon G \to G / H$ be the quotient map.
Choose a continuous function $f \colon G \to [0, 1]$
with compact support and such that $f = 1$ on~$L$.
For $s \in G$ define
\[
g_0 (s) = \int_H f (s h) \, d \mu (h).
\]
Then $g_0$ is continuous and satisfies $g_0 (s k) = g_0 (s)$
for all $s \in G$ and $k \in H$.
Therefore $g_0$ drops
to a continuous function~$g$ on $G / H$.
If $s \not\in \supp (f) H$,
then $g_0 (s) = 0$.
Therefore $\supp ( g ) \subset q ( \supp (f) )$,
and so is compact.
Now
\[
\sup_{s \in G} \mu (s L \cap H)
 \leq \sup_{s \in G} \int_H f (s^{-1} h) \, d \mu (h)
 = \sup_{x \in G / H} g (x)
 < \infty.
\]
This completes the proof.
\end{proof}


\begin{proof}[Proof of \autoref{P_LUCInd-Exact}]
To simplify the notation,
we abbreviate the functor $\LUCInd_H^G$ to~$\LUCInd$.
As in the proof of \autoref{P_Ind_Exact},
it is easy to check that the induced sequence
is exact at the left and
in the middle.
Thus, it remains to check that
$\LUCInd (\pi) \colon \LUCInd (A) \to \LUCInd (B)$ is surjective.

If $G / H$ is compact,
the right uniform induction functor
agrees with the usual induction functor
by \autoref{L_2X10_PpOfLUCInd}\eqref{L_2X10_PpOfLUCInd_Cpt},
so is exact by \autoref{P_Ind_Exact}.

Now assume that $H$ is compact.
It is clear from \autoref{R_NCptInd}
and \autoref{L_2X10_CtUCt}
that
$\LUCInd (A)$ is the fixed point algebra
of the action $\gm \colon H \to \Aut (\Cru (G, A))$
given by
$\gm_h (f) (s) = \af_h (f (s h))$
for $f \in \Cru (G, A)$, $s \in G$, and $h \in H$,
and similarly for~$B$.
By \autoref{P_UniformLifts} and \autoref{P_Gp_LeftU},
the induced \stHm{}
$\kp \colon \Cru (G, A) \to \Cru (G, B)$ is surjective.
Since $H$ is compact, \autoref{P_FixToFix}
implies that the restriction to the fixed point algebras
is also surjective.

For the last case, let us assume that $G$ is a \SINgp.
Let $\af \colon H \to \Aut (A)$ and
$\bt \colon H \to \Aut (B)$ denote the actions of $G$.
Let $q \colon G \to G / H$ denote the quotient map.
Let $\mu$ be a left Haar measure on~$H$.

Let $b \in \LUCInd (B)$ and let $\ep > 0$.
We construct $a \in \LUCInd (A)$ such that
$\| \pi \circ a - b \| < \ep$.
The function $b$ is right uniformly continuous
by \autoref{L_2X10_PpOfLUCInd}(\ref{L_2X10_PpOfLUCInd_Int}).
The hypothesis on $G$
implies that $b$ is left uniformly continuous.
So \autoref{L_2X10_CharRU}
provides an open neighborhood $U$ of $1 \in G$
such that $t^{- 1} s \in U$
implies $\| b (s) - b (t) \| < \frac{\ep}{2}$.
Since $G$ is locally compact,
we may assume that $\overline{U}$ is compact.
Let $V_0$ be an open neighborhood of~$1$
such that ${\overline{V_0}} \subset U$.

We claim that there is
a continuous function $f \colon G \to [0, \infty)$
such that $\supp (f) \subset U$,
such that for every $s \in V_0 H$ we have
\begin{equation}\label{Eq_2X15Intf}
\int_H f (s h) \, d \mu (h) = 1,
\end{equation}
and such that for every $s \in G$ we have
\begin{equation}\label{Eq_3208Intf}
\int_H f (s h) \, d \mu (h) \leq 1.
\end{equation}

We prove the claim.
Choose an open set $Z \subset G$ with
${\overline{V_0}} \subset Z \subset {\overline{Z}} \subset U$,
and choose $f_0 \in C_{\mathrm{c}} (G)$
such that
\[
0 \leq f_0 \leq 1,
\,\,\,\,\,\,
\supp (f_0) \subset U,
\andeqn
f_0 |_{{\overline{Z}}} = 1.
\]
Since $q \big( {\overline{V_0}} \big)$ is compact,
$q (Z)$ is open,
and $q \big( {\overline{V_0}} \big) \subset q (Z)$,
there exists $f_1 \in C_{\mathrm{c}} (G / H)$
such that
\[
0 \leq f_1 \leq 1,
\,\,\,\,\,\,
\supp (f_1) \subset q (Z),
\andeqn
f_1 |_{q ( {\overline{V_0}} )} = 1.
\]
Define a continuous function $k \colon G \to [0, \infty)$
by
\[
k (s) = \int_H f_0 (s h) \, d \mu (h)
\]
for $s \in G$.
For $s \in Z$,
the integrand is equal to~$1$
on the open set $H \cap s^{- 1} Z \subset H$.
This set contains~$1$,
so is nonempty,
whence $k (s) \neq 0$.
Since also $k (s h) = k (s)$ for all $s \in G$ and $h \in H$,
we see that $k (s) \neq 0$ for all $s \in Z H$.
Therefore the definition
\[
f (s)
 = \begin{cases}
   f_1 (s H) f_0 (s) k (s)^{-1} & s \in Z H
       \\
   0 & s \in G \setminus q^{-1} ( \supp (f_1) )
\end{cases}
\]
is consistent
and, since $q^{-1} ( \supp (f_1) ) \subseteq Z H$,
gives a continuous function $f \colon G \to [0, \infty)$.
For $s \in G$,
by considering the cases
$s \in Z H$ and $s \not\in q^{-1} ( \supp (f_1) )$ separately,
one checks that $\int_H f (s h) \, d \mu (h) = f_1 (s H)$.
So (\ref{Eq_2X15Intf}) holds for $s \in V_0 H$
and~(\ref{Eq_3208Intf}) holds for $s \in G$.
This proves the claim.

Since the left and right uniformities on~$G$ agree,
4.14(g) in Chapter~II of \cite{HewRos1979}
provides an open neighborhood $V_1$ of~$1$
such that $s V_1 s^{-1} \subset V_0$ for all $s \in G$.
This implies,
in particular,
that
\begin{equation}\label{Eq_2X14HVH}
H V_1 H \subset V_0 H.
\end{equation}
Now choose an open neighborhood $V$ of~$1$
such that $s, t \in V$ imply $s^{-1} t \in V_1$.

Consider the left uniform cover
${\mathcal{V}} = \{ s V \colon s \in G \}$ of $G$,
and its image
$q ({\mathcal{V}}) = \{ (s V H) / H \colon s \in G \}$ in $G / H$.
Since the left and right uniformities on $G$ agree,
${\mathcal{V}}$ is a right uniform cover of~$G$,
so that $q ({\mathcal{V}})$ is a right uniform cover of $G / H$.
Since $G / H$ is right uniformly finitistic,
there exists a right uniform cover
${\mathcal{W}}$ of $G / H$ which refines $q ({\mathcal{V}})$
and has finite order~$n$.
Let $(l_W)_{W \in {\mathcal{W}}}$
be a right equiuniformly continuous partition of unity on $G / H$
for ${\mathcal{W}}$ as in \autoref{T_2X09UPartU}.
Then the functions $l_W \circ q$
define an equiuniformly continuous partition of unity on~$G$
such that $(l_W \circ q) (x) = 0$ whenever $W \in {\mathcal{W}}$
and $x \in G \setminus q^{-1} (W)$.

For each $W \in {\mathcal{W}}$,
choose a point $x_W \in q^{- 1} (W)$.
Define a continuous function $g_W \colon G \to [0, \infty)$
by
\[
g_W (s) = l_W (s H) \cdot f (x_W^{- 1} s).
\]
This function vanishes outside the set $x_W U \cap q^{- 1} (W)$.
In particular,
$\supp (g_W)$ is contained in the compact set $x_W {\overline{U}}$.

We claim that for every $s \in G$ and $W \in {\mathcal{W}}$, we have
\begin{equation}\label{Eq_2X13gWEst}
\int_H g_W (s h) \, d \mu (h) = l_W (s H).
\end{equation}
If $s \not\in q^{-1} (W)$,
then both sides of~(\ref{Eq_2X13gWEst}) are zero.
To prove the claim,
we therefore assume $s \in q^{-1} (W)$.
Choose $t \in G$ such that $q^{-1} (W) \subset t V H$.
Then $s, x_W \in t V H$,
so there exist $h, k \in H$ such that
$t^{-1} s h, \, t^{-1} x_W k \in V$.
So $k^{-1} x_W^{-1} s h \in V_1$.
It follows from~(\ref{Eq_2X14HVH})
that $x_W^{-1} s \in V_0 H$,
and from~(\ref{Eq_2X15Intf}) that
\[
\int_H f (x_W^{- 1} s h)  \, d \mu (h) = 1.
\]
The claim follows.

For $W \in {\mathcal{W}}$,
choose $a_W \in A$ such that $\pi (a_W) = b (x_W)$ and
$\| a_W \| = \| b (x_W) \|$.
We next claim that the definition
\begin{equation}
\label{Eq:1:P:LUCInd-Exact}
a (s) = \sum_{W \in {\mathcal{W}}}
      \int_H g_W (s h) \cdot \af_h (a_W) \, d \mu (h),
\end{equation}
for $s \in G$, gives a well defined function $a \colon G \to A$.
For each $W \in {\mathcal{W}}$,
the integral exists because the integrand
is continuous and has compact support.
Moreover, for every $s \in G$, from~(\ref{Eq_2X13gWEst}) we get
\begin{equation}
\label{Eq:2:P:LUCInd-Exact}
\left\| \int_H g_W (s h) \cdot \af_h (a_W) \, d \mu (h) \right\|
  \leq l_W (s H) \cdot \| a_W \|
  \leq l_W (s H) \cdot \| b \|.
\end{equation}
It follows that for each $s \in G$ at most $n$~summands
in~\eqref{Eq:1:P:LUCInd-Exact} are nonzero.
The claim follows.
Moreover, $\| a (s) \| \leq \| b \|$ for all $s \in G$.

We claim that $a$ is right uniformly continuous.
Since the left and right uniformities agree,
it suffices to prove that $a$ is left uniformly continuous.
Let $\rh > 0$.
By \autoref{L_2X15MeasCpt},
there is $M > 0$ such that
$\mu \big( t {\overline{U}} \cap H \big) \leq M$
for all $t \in G$.
Using equiuniform continuity of $( l_W \circ q )_{W \in {\mathcal{W}}}$,
choose an open neighborhood $Z_1$ of~$1$ such that
for every $W \in {\mathcal{W}}$
and $s, t \in G$ with $t^{-1} s \in Z_1$,
we have
\[
| l_W (s H) - l_W (t H) | < \frac{\rh}{4 n \| b \| + 1}.
\]
Using \autoref{L_2X13CcUnifC},
choose an open neighborhood $Z_2$ of~$1$
such that whenever $g, h, s, t \in G$
satisfy $s^{-1} t \in Z_2$,
then
\begin{equation}\label{Eq_2X13gsh}
| f (g s h) - f (g t h) | < \frac{\rh}{4 M \| b \| + 1}.
\end{equation}
Define $Z_0 = Z_1 \cap Z_2$.

Now let $s, t \in G$ satisfy $s^{-1} t \in Z_0$.
Then, using $\| a_W \| \leq \| b \|$ for all $W \in {\mathcal{W}}$,
\begin{align*}
\| a (s) - a (t) \|
& = \Bigg\|
    \sum_{W \in {\mathcal{W}}}
     \Bigg( \int_H l_W (s H) f (x_W^{- 1} s h) \af_h (a_W) \, d \mu (h)
     \\
& \hspace*{5em} {\mbox{}}
       - \int_H l_W (t H) f (x_W^{- 1} t h) \af_h (a_W) \, d \mu (h)
       \Bigg) \Bigg\|
     \\
& \leq \| b \| \sum_{W \in {\mathcal{W}}} | l_W (s H) - l_W (t H) |
       \int_H  f (x_W^{- 1} s h) \, d \mu (h)
     \\
& \hspace*{5em} {\mbox{}}
       + \| b \| \sum_{W \in {\mathcal{W}}} l_W (t H)
           \int_H | f (x_W^{- 1} s h) - f (x_W^{- 1} t h) |
            \, d \mu (h).
\end{align*}
In the first term of the last expression,
as in the proof of \autoref{P_UniformLifts},
for any fixed $s, t \in G$,
at most $2 n$ of the terms are nonzero.
Therefore, using~(\ref{Eq_3208Intf}), this term is dominated by
\[
\| b \| \cdot 2 n \left( \frac{\rh}{4 n \| b \| + 1} \right)
     \left( \sup_{W \in {\mathcal{W}}}
           \int_H  f (x_W^{- 1} s h) \, d \mu (h) \right)
  \leq \left( \frac{2 n \| b \| \rh}{4 n \| b \| + 1} \right)
        \cdot 1
  < \frac{\rh}{2}.
\]
Using $\sum_{W \in {\mathcal{W}}} l_W (t H) = 1$,
the choice of~$M$,
and~(\ref{Eq_2X13gsh}),
we see that the second term is dominated by
\[
\| b \| \left( \frac{\rh}{4 M \| b \| + 1} \right)
  \big[ \mu \big( s^{-1} x_W {\overline{U}} \cap H \big)
     + \mu \big( t^{-1} x_W {\overline{U}} \cap H \big) \big]
 \leq \frac{2 M \| b \| \rh}{4 M \| b \| + 1}
 < \frac{\rh}{2}.
\]
So $\| a (s) - a (t) \| < \rh$.
This completes the proof of the claim.

We now claim that $a \in \LUCInd (A)$.
Let $s \in G$ and let $k \in H$.
Using left invariance of~$\mu$ at
the last step,
we get
\begin{align*}
\af_k (a (s k))
& = \af_k \left( \sum_{W \in {\mathcal{W}}}
       \int_H g_W (s k h) \cdot \af_h (a_W) \, d \mu (h) \right) \\
& = \sum_{W \in {\mathcal{W}}}
       \int_H g_W (s k h) \cdot \af_{k h} (a_W) \, d \mu (h)
 = a (s).
\end{align*}
The claim is proved.

It remains to show that $\| \pi \circ a - b \| < \ep$.
Let $s \in G$.
For $W \in {\mathcal{W}}$,
we have constructed $g_W$ such that
if $h \in H$ and $g_W (s h) \neq 0$,
then $x_W^{-1} s h \in U$.
For such $h$ we have
$\| b (x_W) - b (s h) \| < \frac{\ep}{2}$ by the choice of $U$.
Using $H$-equivariance of $\pi$ for the first equality
and $b \in \LUCInd (B)$ for the third equality,
we then get
\begin{align*}
\| \pi ( \af_h (a_W)) - b (s) \|
& = \| \bt_h (b (x_W)) - b (s) \|
  = \| b (x_W) - \bt_{h^{- 1}} (b (s)) \|
   \\
& = \| b (x_W) - b (s h) \|
  < \frac{\ep}{2}.
\end{align*}
Therefore, using~(\ref{Eq_2X13gWEst})
and $\sum_{W \in {\mathcal{W}}} l_W (s H) = 1$
at the first and last steps,
\begin{align*}
\| \pi (a (s)) - b (s) \|
& = \left\| \sum_{W \in {\mathcal{W}}} \int_H g_W (s h)
   \big( \pi ( \af_h (a_W)) - b (s) \big) \, d \mu (h) \right\|
    \\
& \leq \frac{\ep}{2}
     \sum_{W \in {\mathcal{W}}} \int_H g_W (s h) \, d \mu (h)
  < \ep,
\end{align*}
as desired.
\end{proof}


\begin{thm}
\label{P_PjSubGp}
Let $G$ be a \lcg, and let $H \leqslant G$ be a closed subgroup.
Suppose that whenever $\ph \colon A \to B$
is a surjective $H$-morphism of \ca s,
then $\LUCInd_H^G (\ph)$ is also surjective.
Let $\af$ be a \pj{} action of $G$.
Then $\af |_H$ is also \pj.
\end{thm}


\begin{proof}
Let $\bt \colon H \to \Aut (B)$ be an action of $H$ on a \ca~$B$.
The maps to be introduced are shown in the diagram below.
Let $J$ be an $H$-invariant ideal in~$B$,
and let $\kp \colon B \to B / J$ be the quotient map.
Let $\ph \colon A \to B / J$ be an $H$-morphism.
Then $\LUCInd_H^G (\kp) \colon \LUCInd_H^G (B) \to \LUCInd_H^G (B / J)$
is surjective by hypothesis.
We can still define $\et \colon A \to \LUCInd_H^G (B)$
by the same formula as in \autoref{P_Maps},
and it is still a $G$-morphism.
It is easy to check that its range,
which a priori is in $F_H^G (B)$,
is actually in $\LUCInd_H^G (B)$.
Since $\af$ is \pj,
there is a $G$-morphism $\ld \colon A \to \LUCInd_H^G (B)$
such that $\LUCInd_H^G (\kp) \circ \ld = \et$.
We still have $H$-equivariant maps
$\ev_1^B \colon \LUCInd_H^G (B) \to B$
and $\ev_1^{B / J} \colon \LUCInd_H^G (B / J) \to B / J$,
given by the same formulas as in \autoref{P_Eval},
which give the following commutative diagram:
\[
\xymatrix{
& & \LUCInd_H^G (B)
      \ar[d]^{\LUCInd_H^G (\kp)} \ar[rr]^(0.55){\ev_1^B}
   & & B \ar[d]^{\kp} \\
A \ar[rr]_-{\et} \ar[rru]^{\ld}
   & & \LUCInd_H^G (B / J)
      \ar[rr]^(0.55){\ev_1^{B / J}}
   & & B / J.
}
\]
It is easy to check that $\ev_1^{B / J} \circ \et = \ph$.
Therefore the map $\ps = \ev_1^{B} \circ \ld$
is a $H$-morphism from $A$ to~$B$ such that $\kp \circ \ps = \ph$.
This completes the proof that $\af |_H$ is \pj.
\end{proof}


\begin{thm}
\label{P_2X14PjSubGp}
Let $G$ be a \lcg, and let $H \leqslant G$ be a closed subgroup.
Assume that at least one of the following conditions is satisfied:
\begin{enumerate}
\item\label{P_2X14PjSubGp_Cpt}
$H$ is compact.
\item\label{P_2X14PjSubGp_CoCpt}
$G / H$ is compact.
\item\label{P_2X14PjSubGp_SIN}
$G$ is a \SINgp.
\end{enumerate}
Let $\af$ be a \pj{} action of $G$.
Then $\af |_H$ is also \pj.
\end{thm}


\begin{proof}
Combine \autoref{P_LUCInd-Exact}
and \autoref{P_PjSubGp}.
\end{proof}


We point out that \autoref{P_2X14PjSubGp}(\ref{P_2X14PjSubGp_SIN})
applies whenever $G$ is the product of a discrete group
and a locally compact abelian group.


\begin{cor}
\label{P_EqPj_implies_Pj}
Let $G$ be a \lcg, and let $A$ be a $G$-algebra which is \eqpj.
Then $A$ is (non\eqv{ly}) \pj.
\end{cor}


\begin{rmk}\label{R:PjDiscreteSubGp}
\autoref{P_EqPj_implies_Pj}
implies that there is no projective action
of a \lcg{} on a non\pj{} \ca.
This is in contrast to \autoref{R:SjSubGp_notCpct},
where it is shown that the discrete group $\Z$ can act \sj{ly}
on a \ca{} which is not \sj{} in the usual sense.
\end{rmk}


\begin{rmk}
\label{R:NoDLim}
The proof of \autoref{P_PjSubGp}
cannot be generalized to cover semiprojectivity.
This is clear from \autoref{R:SjSubGp_notCpct}.
The problem is that there is no analog of \autoref{P_Ind_Cts}
for the right uniform induction functor.

Let $\N^{+} = \{ 1, 2, \ldots, \infty \}$
be the one point compactification of~$\N$.
Set $B = C (\N^{+})$,
and for $n \in \N$ set
\[
J_n = \big\{ b \in B \colon
  {\mbox{$b (k) = 0$ for $k \in \{ n + 1, \, n + 2, \ldots, \infty \}$}}
     \big\}.
\]
Then
${\overline{\bigcup_{n = 1}^{\infty} J_n}} = C_0 (\N) \subset B$.
Call this ideal~$J$.
For $l \in \N$,
define $b_l \in B$ by
\[
b_l (j) = \begin{cases}
   1 & j = l
       \\
   0 & j \neq l,
\end{cases}
\]
and define $a \in C_{\mathrm{b}} (\Z, B)$ by
$a (n) = b_n$ for $n \in \N$
and $a (n) = 0$ for $n \in \Z \setminus \N$.
Then $a \in C_{\mathrm{b}} (\Z, J)$,
but the distance from $a$ to any element of
$\bigcup_{n = 1}^{\infty} C_{\mathrm{b}} (\Z, J_n)$
is at least~$1$,
so
$a \not \in
   {\overline{\bigcup_{n = 1}^{\infty} C_{\mathrm{b}} (\Z, J_n)}}$.

We have written everything in terms of bounded continuous functions,
but on $\Z$ all continuous functions are uniformly continuous.
\end{rmk}



\section{Semiprojectivity of the crossed product algebra}
\label{Sec_SjCrPrd}

\indent
If $\GAa$ is an \eqsj{} \ca,
can we deduce that the crossed product algebra $\CGAa$
is \sj?
We show in \autoref{P_CrPrdSj}
that the answer is yes when $G$ is finite and $A$ is unital,
and in \autoref{E_2X10_CptCP} that the answer can be no
when $G$ is compact.
We then provide examples to show that the converses
of both \autoref{P_CrPrdSj} and \autoref{P_EqSj_implies_Sj} are false.
We end the section with further open problems.


\begin{thm}
\label{P_CrPrdSj}
Let $G$ be a discrete group such that $C^* (G)$ is \sj,
and let $\GAa$ be a unital \ga{}
which is \eqsj{} in the unital category.
Then $\CGAa$ is \sj{} (in the usual sense).
\end{thm}


\begin{proof}
We will show that $\CGAa$ is \sj{}
in the unital category.
Applying \autoref{P_SjUnitDiffCats},
with the group being trivial,
we conclude that $\CGAa$ is \sj.

We regard $A$ as a subalgebra of $\CGAa$.
Also, for $s \in G$ let $u_s \in \CGAa$ be the standard implementing
unitary,
so that $u_s a u_s^* = \af_s (a)$ for all $a \in A$.
The unitaries $u_s$ induce a \stHm{}
$\om \colon C^* (G) \to \CGAa$.

By assumption, $C^* (G)$ is \sj.
Thus, Lemma~1.4 of \cite{Phi2012} shows that it suffices to prove
that $\om$ is relatively \sj{}
in the sense of Definition~1.2 of \cite{Phi2012}
(but with the group being trivial).
Accordingly,
let $C$ be a unital \ca,
let $J_1 \subset J_2 \subset \cdots$ be ideals in~$C$,
let $J = {\overline{\bigcup_{n = 1}^{\infty} J_n}}$,
let
\[
\kp \colon C \to C / J,
\, \, \, \, \, \,
\kp_n \colon C \to C / J_n,
\andeqn
\pi_n \colon C / J_n \to C / J
\]
be the quotient maps,
and let $\ld \colon C^* (G) \to C$
and $\ph \colon \CGAa \to C / J$
be unital \stHm{s}
such that $\kp \circ \ld = \ph \circ \om$.

Define an action $\gm \colon G \to \Aut (C)$
by $\gm_s (c) = \ld (u_s) c \ld (u_s)^*$
for $c \in C$ and $s \in G$.
Then $(G, C, \gm)$ is a unital \ga,
and the ideals $J_n$ are $G$-invariant.

One checks that $\ph |_A \colon A \to C/J$ is $G$-equivariant.
Since $\GAa$ is \eqsj{} (in the unital category),
there exists $n \in \N$
and a unital $G$-morphism $\ps_0 \colon A \to C / J_n$
such that $\pi_n \circ \ps_0 = \ph |_A$.
Define $v_s = (\kp_n \circ \ld) (u_s)$ for $s \in G$.
Then $(v, \ps_0)$ is a covariant representation of $\GAa$ in $C / J_n$,
so there exists a unique \stHm{} $\ps \colon \CGAa \to C / J_n$
such that $\ps (u_s) = v_s$ and $\ps |_A = \ps_0$.
This \stHm{} is the one required by the definition of
relative semiprojectivity.
\end{proof}


The basic examples of countable discrete groups~$G$
that satisfy the hypothesis
of \autoref{P_CrPrdSj}, that is, such that $C^* (G)$ is \sj,
are finite groups, $\Z$, and the finitely generated free groups.
There is no known characterization
of those
groups $G$ for which $C^* (G)$ is \sj.

In \autoref{P_CrPrdSj}, some restriction on $G$ is necessary.
Even compactness is not enough.


\begin{exa}\label{E_2X10_CptCP}
Let $G$ be an infinite compact group.
It follows from Corollary~1.9 of \cite{Phi2012}
that the trivial action of $G$ on~$\C$ is \sj.
However, the crossed product is $\C \rtimes G = C^* (G)$,
which is an infinite direct sum of matrix algebras,
so not \sj{} by Corollary~2.10 of \cite{Bla2004}.
\end{exa}


\autoref{P_CrPrdSj} gives us an easy way of proving
that many actions by $\Z$ are not \eqsj{}
(in the unital category).


\begin{exa}\label{E_2X10_RotNSj}
Let $\te \in \R$.
Let $\af \colon \Z \to \Aut (C (S^1) )$
be the action generated by rotation by $\exp (2 \pi i \te)$.
Then $\af$ is never \sj{} in the unital category,
for any value of~$\te$.

If $\te \not\in \Q$,
then the crossed product is a simple $A {\mathbb{T}}$-algebra,
and therefore not \sj,
for example by Corollary~2.14 of \cite{Bla2004}.

If $\te \in \Q$,
then $A = C (S^1) \rtimes_{\af} \Z$
is Morita equivalent to $C \big( (S^1)^2 \big)$.
Since both $A$ and $C \big( (S^1)^2 \big)$ are unital and
$C \big( (S^1)^2 \big)$ is not \sj,
it follows from Corollary~2.29 of \cite{Bla1985}
that $A$ is not \sj.

In both cases, it follows from \autoref{P_CrPrdSj}
that $\af$ is not \eqsj.
\end{exa}


There are versions of \autoref{P_CrPrdSj} in which one
takes the crossed product by only part of the action.
As an easy example,
consider an action of a product of two groups,
and take the crossed product by one of them.
We will not explore the possibilities further here.


We end this section with two examples that show that the converses
of both \autoref{P_CrPrdSj} and \autoref{P_EqSj_implies_Sj} are false,
and we give more open problems.


\begin{exa}
\label{E:CrPrdNotSj}
There is an action $\af$ of $\Z_2$ on $\OA{2}$ such that
the crossed product $B = \OA{2} \rtimes_\af \Z_2$ is not \sj.
It follows from \autoref{P_CrPrdSj} that
this action is not \eqsj.
Thus, the converse of \autoref{P_EqSj_implies_Sj} fails.

We follow \cite{Izu2004};
also see Section~6 of \cite{Bla2004}.
Take $\af$ to be as in Lemma~4.7 of \cite{Izu2004}
or, more generally,
as in Theorem 4.8(3) of \cite{Izu2004}
with the groups $\Gm_0$ and~$\Gm_1$ chosen so that at least
one of them is not finitely generated,
and also such that $\OA{2} \rtimes_\af \Z_2$
satisfies the Universal Coefficient Theorem.
The action $\af$ is outer,
so $B$ is simple by Theorem~3.1 of \cite{Kis1981}
and purely infinite by Corollary~4.6 of \cite{JEOOSA1998}.
Therefore it is a Kirchberg algebra
(a separable purely infinite simple nuclear \ca).
It does not have finitely generated K-theory,
so $B$ is not \sj{} by Corollary~2.11 of \cite{Bla2004}.
\end{exa}


\begin{exa}
\label{E:AlgNotSj}
Let ${\widehat{\af}} \colon \Z_2 \to \Aut (B)$
be the dual of the action $\af$ of \autoref{E:CrPrdNotSj}.
Then
\[
B \rtimes_{\widehat{\af}} \Z_2
\cong M_2 \otimes \OA{2}
\cong \OA{2},
\]
which is \sj.
However, $B$ was shown in \autoref{E:CrPrdNotSj}
not to be \sj.
So \autoref{P_EqSj_implies_Sj} implies
that $\widehat{\af}$ is not \eqsj.
This shows that the converse of \autoref{P_CrPrdSj} fails.
\end{exa}


\autoref{E:CrPrdNotSj}
also shows if $A$ is \sj{} and $\aGA$ is an action
of a finite group on~$A$,
then $\GAa$ need not be \eqsj.
However, we have neither a proof nor a counterexample for the
following question.


\begin{qst}
\label{Q:IfAAndCStSj}
Let $G$ be a finite cyclic group of prime order,
and let $\GAa$ be a $G$-algebra.
Suppose that $A$ and $\CGAa$ are both \sj.
Does it follow that $\GAa$ is \eqsj?
\end{qst}


If $\aGA$ is \sj, then \autoref{P_SjSubGp} implies
that for any subgroup $H \leqslant G$,
the action $\alpha \vert_H$ is also \sj.
Thus, by \autoref{P_CrPrdSj}, the crossed product
$A \rtimes_{\af \vert_H} H$ is semiprojective.
If $G$ has proper subgroups,
one must therefore probably also consider
these intermediate crossed product algebras.

At a conference in August 2010,
George Elliott asked if there is a relation between
\eqv{} semiprojectivity and the Rokhlin property.
The following question addresses what seems to be a plausible
connection.


\begin{qst}
\label{Q:SjAndRP}
Let $G$ be a finite group,
and let $\GAa$ be a unital \ga.
Suppose that $A$ is (non\eqv{ly}) \sj{}
and $\af$ has the Rokhlin property.
Does it follow that $\GAa$ is \eqsj?
\end{qst}


Even if this is false in general,
it might be true if $A$ is simple,
or using an equivariant version of a weak form of semiprojectivity.


\section{Semiprojectivity of the fixed point algebra}\label{Sec_SjFixPt}


In this section we study the analog
of the question of \autoref{Sec_SjCrPrd}
for the fixed point algebra.
That is,
given an \eqsj{} \ca{} $\GAa$,
can we deduce that the fixed point algebra $A^G$ is \sj?

In \autoref{P_2X10_SatSj},
we give a positive answer when $G$ is finite,
$A$ is unital, and the action is saturated.
We do not know whether one can drop the conditions that $A$ be unital
or that the action be saturated.

Some conditions are necessary.
In \autoref{E_2X10_CptFix} we give a \sj{} action
of a compact (but not finite) group on a unital \ca\
such that the fixed point algebra is not \sj.

In \autoref{P_FixPt},
we show that if a noncompact group acts \sj{ly}
then the fixed point algebra is trivial.
This gives a positive answer to the question,
but more interestingly it shows that the trivial action
of a noncompact group is never \sj.
We can therefore give a precise characterization
when the trivial action of a group is \sjpj{} (\autoref{P_SjTriv}).


Let $G$ be a second countable compact group
and let $\af \colon G \to \Aut (A)$ be a \sj{} action.
\autoref{E_2X10_CptCP} shows that the crossed product
$A \rtimes_{\af} G$ need not be \sj,
but in that example the fixed point algebra is \sj.
In general, though, the fixed point algebra also need not be \sj.


\begin{exa}\label{E_2X10_CptFix}
Let $\af \colon S^1 \to \Aut ( \OA{2} )$
be the gauge action on the Cuntz algebra~$\OA{2}$,
defined on the standard generators $s_1$ and $s_2$ by
$\af_{\zt} (s_j) = \zt s_j$ for $\zt \in S^1$ and $j = 1, 2$.
This action is \eqsj{} by Corollary~3.12 of \cite{Phi2012}.
However, the fixed point algebra is the $2^{\infty}$~UHF algebra,
which is not \sj, for example by Corollary~2.14 of \cite{Bla2004}.
\end{exa}


We obtain a positive result when the group is finite
and the action is saturated
in the sense of Definition 7.1.4 of \cite{Phi1987}.
Saturation is a quite weak noncommutative analog of freeness;
see the discussion at the beginning of Section~5.2 of \cite{Phi2009}.


\begin{prp}\label{P_2X10_SatSj}
Let $G$ be a finite group,
let $A$ be a unital, separable \ca,
and let $\aGA$ be a saturated action of $G$ on~$A$.
If $\af$ is \sj,
then $A^G$ is \sj.
\end{prp}


\begin{proof}
By definition,
saturation implies that $A^G$ is strongly Morita equivalent
to $\CGAa$.
\autoref{P_CrPrdSj} tells us that $\CGAa$ is \sj,
so $A^G$ is \sj{} by Corollary~2.29 of \cite{Bla1985}.
\end{proof}


Finiteness is needed,
since the gauge action in \autoref{E_2X10_CptFix}
is saturated.
(In fact, it follows from Theorem~5.11 of \cite{Phi2009}
that this action is hereditarily saturated.)
However, we don't know whether saturation is needed.


\begin{qst}\label{Q_2X10_FixFinGp}
Let $G$ be a finite group,
let $A$ be a unital \ca,
and let $\aGA$ be an arbitrary \sj{} action of $G$ on~$A$.
Does it follow that $A^G$ is \sj?
\end{qst}


If $G$ is compact and $A$ unital,
then $A^G$ is isomorphic to a unital corner in $\CGAa$,
for example by Theorem~II.10.4.18 of \cite{Bla2006}.
If we knew that
semiprojectivity passes to arbitrary unital corners
(an open problem),
we would get a positive answer to \autoref{Q_2X10_FixFinGp}.


\begin{thm}
\label{P_FixPt}
Let $\GAa$ be a separable \eqsj{} \ga,
and assume that $G$ is not compact.
Then $A^G = \{ 0 \}$.
\end{thm}


\begin{proof}
We manufacture an equivariant lifting problem in several steps.

Step~1:
The action $\af \colon G \to \Aut(A)$ induces
an action ${\overline{\af}} \colon G \to \Aut( M_2 \otimes A )$
by acting trivially on $M_2$, that is,
${\overline{\af}}_s ( x \otimes a ) = x \otimes \af_s(a)$
for $x \in M_2$, $a \in A$, and $s \in G$.
Let $(e_{j, k})_{j, k = 1, 2}$
be the standard system of matrix units for~$M_2$.
Let $\ld \mapsto u_{\ld} \in M_2$,
for $\ld \in [0, 1]$,
be a continuously differentiable path of unitaries from the identity
$u_0 = \left(\begin{smallmatrix}
1 & 0 \\ 0 & 1
\end{smallmatrix} \right)$
to
$u_1 = \left(\begin{smallmatrix}
0 & 1 \\ 1 & 0
\end{smallmatrix} \right)$.
Continuous differentiability is required for convenience;
it gives us $M \in [0, \infty)$
such that for all $\ld_1, \ld_2 \in [0, 1]$
we have $\| u_{\ld_1} - u_{\ld_2} \| \leq M | \ld_1 - \ld_2 |$.
For $\ld \in [0, 1]$ define $\ph_{\ld} \colon A \to M_2 \otimes A$
by
$\ph_{\ld} (a) = u_{\ld} e_{1, 1} u_{\ld}^* \otimes a$
for $a \in A$.
Thus $\ph_0 (a) = e_{1, 1} \otimes a$
and $\ph_1 (a) = e_{2, 2} \otimes a$
for $a \in A$.
Also, for $\ld_1, \ld_2 \in [0, 1]$ and $a \in A$,
we have
\begin{equation}\label{Eq_2X09_St}
\| \ph_{\ld_1} (a) - \ph_{\ld_2} (a) \|
  \leq 2 \| u_{\ld_1} - u_{\ld_2} \| \cdot \| a \|
  \leq 2 M | \ld_1 - \ld_2 | \cdot \| a \|.
\end{equation}
It is immediate that
\begin{equation}\label{Eq:P_FixPt:St2.2}
{\overline{\af}}_s \circ \ph_\ld = \ph_\ld \circ \af_s.
\end{equation}
for $s \in G$ and $\ld \in [0,1]$.

Step~2:
Let $G^+ = G \cup \{\infty\}$ denote
the one point compactification of~$G$.
Let $D$ be the \ca{}
\[
D = \big\{ f \in C (G^+, \, M_2 \otimes A) \colon
         f (\infty) \in \C e_{2, 2} \otimes A \big\}.
\]
For $s \in G$,
we take $s \cdot \infty = \infty$.
This gives a extension of the action of $G$
on itself by translation
to a continuous action of $G$ on~$G^{+}$.
We define an action $\bt$ of $G$ on $D$ by
$\bt_s (f) (t) = {\overline{\af}}_s ( f (s^{- 1} t) )$
for $f \in D$, $s \in G$, and $t \in G^+$.
Since $G$ is not compact,
the fixed point algebra of this action consists of the
constant functions taking values in $\C e_{2, 2} \otimes A^G$.

Step~3:
For $k = 1, 2, \ldots$,
define ``stretching'' maps $\sm_k \colon [0, \infty) \to [0, 1]$
by
\[
\sm_k (\ld) = \min ( \ld / k, \, 1)
\]
for $\ld \in [0, \infty)$.
We may extend these to maps from $[0, \infty]$
by setting $\sm_k (\infty) = 1$ for $k \in \N$.
For $\ld_1, \ld_2 \in [0, \infty)$,
we have
\begin{equation}\label{Eq_2X09_StSt}
| \sm_k (\ld_1) - \sm_k (\ld_2) |
  \leq \frac{| \ld_1 - \ld_2 |}{k}.
\end{equation}
%

Step~4:
Recall that a metric is called {\emph{proper}}
if every closed bounded set is compact.
By the main theorem of \cite{Str1974},
there is a proper
left invariant metric~$d$ which generates the topology of~$G$.
For $t \in G$ let $d_0 (t) = d (t, 1)$ denote the distance from $t$
to the identity $1 \in G$,
and extend this function to $G^+$ by setting $d_0 (\infty) = \infty$.
Since $d$ is proper,
the map $d_0 \colon G^+ \to [0, \infty]$ is continuous.

Using left invariance of~$d$,
for $s, t \in G$ we get
$| d_0 (s^{-1} t) - d_0 (t) | \leq d_0 (s)$.
Therefore, for $k \in \N$,
\begin{equation}\label{Eq_2X09_StStSt}
\big| \sm_k (d_0 (s^{-1} t)) - \sm_k (d_0 (t)) \big|
  \leq \frac{d_0 (s)}{k}.
\end{equation}

For $k \in \N$, we define a \stHm{}
$\om_k \colon A \to D$ by
\[
\om_k (a) (t) = \ph_{\sm_k (d_0 (t))} ( a )
\]
for $a \in A$ and $t \in G^{+}$.
Then for $s \in G$ we have,
using density of $G$ in $G^{+}$
and~(\ref{Eq:P_FixPt:St2.2}) at the second step,
(\ref{Eq_2X09_St})~at the third step,
and (\ref{Eq_2X09_StStSt})~at the fourth step,
\begin{align}\label{Eq_2X09_Q}
\| \bt_s (\om_k (a)) - \om_k ( \af_s (a) ) \|
 & = \sup_{t \in G^{+}}
    \| {\overline{\af}}_s ( \ph_{\sm_k (d_0 (s^{- 1} t))} (a) )
    - \ph_{\sm_k (d_0 (t))} ( \af_s (a) ) \|
        \\
 & = \sup_{t \in G}
    \| \ph_{\sm_k (d_0 (s^{- 1} t))} ( \af_s (a) ) )
    - \ph_{\sm_k (d_0 (t))} ( \af_s (a) ) \|
       \notag
        \\
 & \leq \sup_{t \in G}
     2 M \big| \sm_k (d_0 (s^{-1} t)) - \sm_k (d_0 (t)) \big|
     \cdot \| \af_s (a) \|
       \notag
        \\
 & \leq \frac{2 M \| a \| d_0 (s)}{k}.
       \notag
\end{align}
In particular,
we have
\begin{equation}\label{Eq:P:SjTriv}
\lim_{k \to \infty} \| \bt_s (\om_k (a)) - \om_k ( \af_s (a) ) \|
 = 0.
\end{equation}

Moreover, for $k \in \N$ and $a \in A$,
we have $\om_k (a) (1) = e_{1, 1} \otimes a$,
so, using Step~2,
\begin{align}\label{Eq_2X09_22}
\dist ( \om_k (a), \, D^G)
  & \geq \inf_{b \in A} \| e_{1, 1} \otimes a - e_{2, 2} \otimes b \|
    \\
  & \geq \inf_{b \in A} \| (e_{1, 1} \otimes 1)
             (e_{1, 1} \otimes a - e_{2, 2} \otimes b) \|
   = \| a \|. \nonumber
\end{align}
%

Step~5:
Let $E$ be the sequence algebra $E = l^{\infty} (\N, D)$.
Let $\gm \colon G \to \Aut (E)$ denote the
(not necessarily \ct) coordinatewise action of $G$ on $E$,
that is,
for $s \in G$ and $(x_k)_{k \in \N} \in E$
we set $\gm_s ( (x_k)_{k \in \N} ) = ( \bt_s (x_k) )_{k \in \N}$.
Let $F \subset E$ be the $C^*$-subalgebra
\[
F = \big\{ x \in E \colon
       {\mbox{$s \mapsto \gm_s (x)$ is continuous}} \big\}.
\]
Then $F$ is $\gm$-invariant,
and we also use $\gm$ to denote the restricted action
$\gm \colon G \to \Aut (F)$.
By construction, this action is \ct.

For $n \in \N$ define
\[
J_n = \big\{ (x_k)_{k \in \N} \in E \colon
            {\mbox{$x_k = 0$ for $k \geq n$}} \big\}
     \subset E.
\]
Then $J_1 \subset J_2 \subset \cdots$
is an increasing sequence of invariant ideals,
and the ideal $J = \overline{\bigcup_{n = 1}^{\infty} J_n}$
is equal to $C_0 (\N, D) \subset l^{\infty} (\N, D)$.
Clearly $J \subset F$.
We can identify $F / J_n$
with the set of elements of
$l^{\infty} \big( \{ n + 1, \, n + 2, \ldots \}, \, D \big)$
on which the coordinatewise action of $G$ is \ct.

For $n \in \N$, the action of $G$ on $F$ drops to $F / J_n$,
and the action also drops to $F / J$.
Let $\pi_n \colon F / J_n \to F / J$
be the natural quotient $G$-morphism.
We have $F^G = l^{\infty} (\N, D^G)$,
and one checks by direct computation
that
$(F / J_n)^G = F^G / J_n^G$,
which we identify with
$l^{\infty} \big( \{ n + 1, \, n + 2, \ldots \}, \, D^G \big)$.

Step~6:
For each $a \in A$, consider the sequence
$\om (a) = (\om_1 (a), \, \om_2 (a), \, \ldots) \in E$
constructed in Step~4.
We claim that $\om (a) \in F$.
To see this,
let $a \in A$ and let $s, t \in G$.
Then,
using~(\ref{Eq_2X09_Q}) and $\|\om_k\|=1$ at the third step,
\begin{align*}
\| \gm_s (\om (a)) - \gm_t ( \om (a)) \|
 & = \sup_{k \in \N} \| \bt_s (\om_k (a)) - \bt_t ( \om_k (a)) \|
   \\
 &  = \sup_{k \in \N} \| \bt_{t^{-1} s} (\om_k (a))
 - \om_k ( \af_{t^{-1} s} (a) )
 + \om_k ( \af_{t^{-1} s} (a) - a) \|
      \\
 & \leq \sup_{k \in \N} \frac{2 M \| a \| d_0 (t^{-1} s)}{k}
 + \| \af_{t^{-1} s} (a) - a \|
   \\
 & = 2 M \| a \| d ( s, t) + \| \af_s (a) - \af_t (a) \|.
\end{align*}
Since $\af$ is a continuous action, this proves the claim.

Step~7:
Define a \stHm{} ${\overline{\om}} \colon A \to F / J$
by sending $a \in A$
to the image of $\om (a)$ in the quotient $F / J$.
It follows from~\eqref{Eq:P:SjTriv}
that ${\overline{\om}}$ is a $G$-morphism.

Suppose now that $A$ is \eqsj.
Then there are $n \in \N$ and a $G$-morphism
$\ps \colon A \to F / J_n$ such that
$\pi_n \circ \ps = {\overline{\om}}$.

Fix an element $a \in A^G$.
We want to show $a = 0$.
Since $\ps$ is $G$-equivariant,
$\ps (a) \in (F / J_n)^G$.
Identify $(F / J_n)^G$ with
$l^{\infty} \big( \{ n + 1, \, n + 2, \ldots \}, \, D^G \big)$
as at the end of Step~5,
and write $\ps (a) = ( \ps_{n + 1} (a), \, \ps_{n + 2} (a), \, \ldots)$.
Then, using~(\ref{Eq_2X09_22}) at the last step,
\begin{align*}
\| \pi_n ( \ps (a)) - {\overline{\om}} (a) \|
 & = \big\| \pi_n \big( ( \ps_{n + 1} (a), \, \ps_{n + 2} (a), \,
                    \ldots)
       - ( \om_{n + 1} (a), \, \om_{n + 2} (a), \, \ldots) \big) \big\|
   \\
 & = \liminf_{k \to \infty} \| \ps_k (a) - \om_k (a) \|
   \\
 & \geq \inf_{k \in \{ n + 1, \, n + 2, \, \ldots \} }
             \dist ( \om_k (a), \, D^G )
   \geq \| a \|.
\end{align*}
For $a \neq 0$
this contradicts $\pi_n ( \ps (a)) = {\overline{\om}} (a)$.
Thus $A^G = \{ 0 \}$.
\end{proof}


\begin{cor}\label{P_SjTriv}
Let $A$ be a nonzero separable \ca,
and let $G$ be a second countable \lcg.
Then the trivial action of $G$ on $A$ is \sjpj{}
if and only if $A$ is \sjpj{} and $G$ is compact.
\end{cor}


\begin{proof}
If $G$ is not compact and $\GAa$ is equivariantly semiprojective,
then \autoref{P_FixPt} implies that $A^G = \{ 0 \}$.
If $\af$ is trivial, it follows that $A = \{ 0 \}$.

Now suppose $G$ is compact.
If $A$ is \sjpj,
then it follows from Corollary~1.9 of \cite{Phi2012}
and \autoref{P_PjTrivAct}
that the trivial action of $G$ on $A$ is \sjpj.
The converse follows from \autoref{P_EqPj_implies_Pj}
and \autoref{P_EqSj_implies_Sj}.
\end{proof}

\section*{Acknowledgments}

\indent
The first named author
thanks the Research Institute for Mathematical Sciences
of Kyoto University for its support through a visiting professorship.
The second and third named authors
thank Siegfried Echterhoff and Stefan Wagner
for valuable discussions on induction of group actions.
The third named author thanks Joachim Zacharias for pointing out
the reference \cite{Ank1989}.



\begin{thebibliography}{EKQR00}

\bibitem[Ank89]{Ank1989}
J.~Anker,
{\emph{A short proof of a classical covering lemma}},
Monatsh.\  Math.\  {\textbf{107}} (1989), 5--7.

\bibitem[BG52]{BarGra1952}
R.~G.\  Bartle and L.~M. Graves,
{\emph{Mappings between function spaces}},
Trans.\  Amer.\  Math.\  Soc.\  {\textbf{72}} (1952), 400--413.

\bibitem[Bla85]{Bla1985}
B.~Blackadar,
{\emph{Shape theory for $C^*$-algebras}}, Math.\  Scand.\  %
{\textbf{56}} (1985), 249--275.

\bibitem[Bla04]{Bla2004}
\bysame, {\emph{Semiprojectivity in simple $C^*$-algebras}},
{Kosaki, H.\  (ed.), Operator Algebras and Applications.
Proceedings of the US-Japan
Seminar held at Kyushu University, Fukuoka, Japan, June 7--11, 1999. Tokyo:
Mathematical Society of Japan.
Advanced Studies in Pure Mathematics, 38, 1-17}, 2004.

\bibitem[Bla06]{Bla2006}
\bysame, {\emph{Operator Algebras.
Theory of $C^*$-Algebras and von Neumann Algebras}},
{Encyclopaedia of Mathematical Sciences 122. Operator Algebras and
Non-Commutative Geometry~3. Berlin: Springer}, 2006.

\bibitem[Bla12]{BlaPrivate}
\bysame, private communication, 2012.

\bibitem[Bln96]{Blanch1996}
E.~Blanchard, {\emph{D\'{e}formations de $C^*$-alg\`{e}bres de Hopf}},
Bull.\  Soc.\  Math.\  Fr.\  %
{\textbf{124}} (1996), 141--215.

\bibitem[Bre72]{Bre1972}
G.~E.\  Bredon, {\emph{Introduction to Compact Transformation Groups}}, {Pure and Applied Mathematics, 46.
New York-London: Academic Press}, 1972.

\bibitem[Dad09]{Dad2009}
M.~Dadarlat,
{\emph{Continuous fields of $C^*$-algebras over finite dimensional
spaces}}, Adv.\  Math.\  {\textbf{222}} (2009), 1850--1881.

\bibitem[Ech10]{Ech2010}
S.~Echterhoff,
{\emph{Crossed products, the Mackey-Rieffel-Green machine and
applications}}, preprint, arXiv:1006.4975, 2010.

\bibitem[EKQR00]{EchKalQuiRae2000}
S.~Echterhoff, S.~Kaliszewski, J.~Quigg, and I.~Raeburn, {\emph{Naturality and induced representations}},
Bull.\  Aust.\  Math.\  Soc.\  {\textbf{61}} (2000), 415--438.

\bibitem[HR79]{HewRos1979}
E.~Hewitt and K.~A.\  Ross, {\emph{Abstract Harmonic Analysis.
Vol. 1: Structure of Topological Groups; Integration Theory;
Group Representations}}, 2nd ed.,
{Grundlehren der mathematischen Wissenschaften, 115. A Series of
Comprehensive Studies in Mathematics. Berlin-Heidelberg-New York:
Springer-Verlag}, 1979.

\bibitem[Isb59]{Isb1959}
J.~R.\  Isbell, {\emph{On finite-dimensional uniform spaces}},
Pac.\  J.\  Math.\  {\textbf{9}} (1959), 107--121.

\bibitem[Isb64]{Isb1964}
\bysame, {\emph{Uniform Spaces}},
{Mathematical Surveys, 12. Providence, R.I.:
American Mathematical Society}, 1964.

\bibitem[Izu04]{Izu2004}
M.~Izumi,
{\emph{Finite group actions on $C^*$-algebras
with the Rohlin property. I}},
Duke Math.\  J.\  {\textbf{122}} (2004), 233--280.

\bibitem[JO98]{JEOOSA1998}
J.~A Jeong and H.~Osaka,
{\emph{Extremally rich $C^*$-crossed products and the
 cancellation property}}, J.~Aust.\  Math.\  Soc., Ser.~A
{\textbf{64}} (1998), 285--301.

\bibitem[Kas88]{Kas1988}
G.~G.\  Kasparov,
{\emph{Equivariant KK-theory and the Novikov conjecture}},
  Invent.\  Math.\  {\textbf{91}} (1988), 147--201.

\bibitem[Kis81]{Kis1981}
A.~Kishimoto, {\emph{Outer automorphisms and reduced crossed products
of simple $C^*$-algebras}},
Commun.\  Math.\  Phys.\  {\textbf{81}} (1981), 429--435.

\bibitem[KW99]{KIRWAS1999B}
E.~Kirchberg and S.~Wassermann,
{\emph{Permanence properties of $C^*$-exact groups}},
Doc.\  Math.\  {\textbf{4}} (1999), 513--558.

\bibitem[Lor97a]{Lor1997b}
T.~A.\  Loring,
{\emph{Almost multiplicative maps between $C^*$-algebras}},
{Doplicher, S. et al. (eds.),
Operator Algebras and Quantum Field Theory. Proceedings of
the conference dedicated to Daniel Kastler in celebration of his 70th
birthday, Accademia Nazionale dei Lincei, Roma, Italy, July 1--6, 1996.
Cambridge, MA: International Press. 111-122}, 1997.

\bibitem[Lor97b]{Lor1997}
\bysame,
{\emph{Lifting Solutions to Perturbing Problems in $C^*$-Algebras}},
{Fields Institute Monographs,~8.
Providence, R.I.: American Mathematical Society}, 1997.

\bibitem[MZ55]{MONZIP1955}
D.~Montgomery and L.~Zippin,
{\emph{Topological Transformation Groups}},
{Interscience Tracts in Pure and Applied Mathematics. New York:
Interscience Publishers, Inc.}, 1955.

\bibitem[Pal78]{Pal1978}
T.~W.\  Palmer,
{\emph{Classes of nonabelian, noncompact, locally compact groups}},
Rocky Mountain J.\  Math..\  {\textbf{8}} (1978), 683--741.

\bibitem[Pea75]{Pea1975}
A.~R.\  Pears,
{\emph{Dimension Theory of General Spaces}}, {Cambridge etc.:
Cambridge University Press}, 1975.

\bibitem[Phi87]{Phi1987}
N.~C.\  Phillips,
{\emph{Equivariant K-Theory and Freeness of Group Actions on
$C^*$-Algebras}}, {Lecture Notes in Mathematics, 1274.
Berlin etc.: Springer-Verlag}, 1987.

\bibitem[Phi09]{Phi2009}
\bysame,
{\emph{Freeness of actions of finite groups on $C^*$-algebras}},
{de Jeu, M.\  et al. (eds.),
Operator Structures and Dynamical Systems. Satellite
conference of the 5th European Congress of Mathematics,
Leiden, Netherlands,
July 21--25, 2008. Providence, RI: American Mathematical Society.
Contemporary Mathematics 503, 217-257}, 2009.

\bibitem[Phi12]{Phi2012}
\bysame, {\emph{Equivariant semiprojectivity}},
preprint, arXiv:1112.4584, 2012.

\bibitem[Rie89]{Rie1989}
M.~A.\  Rieffel,
{\emph{Continuous fields of $C^*$-algebras coming from group cocycles
and actions}}, Math.\  Ann.\  {\textbf{283}} (1989), 631--643.

\bibitem[Roy88]{Roy1988}
H.~L.\  Royden, {\emph{Real Analysis}}, 3rd ed.
{New York: Macmillan Publishing
Company; London: Collier Macmillan Publishing}, 1988.

\bibitem[SSG93]{SegSpiGue1993}
J.~Segal, S.~Spie\.{z}, and B.~G\"{u}nther,
{\emph{Strong shape of uniform spaces}}, Topology Appl.\  %
{\textbf{49}} (1993), 237--249.

\bibitem[Str74]{Str1974}
R.~A.\  Struble, {\emph{Metrics in locally compact groups}},
Compos.\  Math.\  {\textbf{28}} (1974), 217--222.

\bibitem[Swa59]{Swa1959}
R.~G. Swan, {\emph{A new method in fixed point theory}},
Bull.\  Amer.\  Math.\  Soc.\  %
{\textbf{65}} (1959), 128--130.

\bibitem[TW12]{ThiWin2012}
H.~Thiel and W.~Winter,
{\emph{The generator problem for ${\mathcal{Z}}$-stable
$C^*$-algebras}},
Trans.\  Amer.\  Math.\  Soc.,
to appear
(arXiv: 1201.3879v1 [math.OA]).

\bibitem[Vid69]{Vid1969}
G.~Vidossich,
{\emph{A theorem on uniformly continuous extension of mappings
defined in finite-dimensional spaces}},
Isr.\  J.\  Math. {\textbf{7}} (1969), 207--210.

\bibitem[Wil07]{Wl} D.~P.\  Williams,
{\emph{Crossed Products of C*-Algebras}},
Mathematical Surveys and Monographs, 134.
Providence, RI: American Mathematical Society, 2007.

\end{thebibliography}

\providecommand{\bysame}{\leavevmode\hbox to3em{\hrulefill}\thinspace}
\providecommand{\MR}{\relax\ifhmode\unskip\space\fi MR }
\providecommand{\MRhref}[2]{%
  \href{http://www.ams.org/mathscinet-getitem?mr=#1}{#2}
}
\providecommand{\href}[2]{#2}

\end{document}